\documentclass[11pt]{article}%
\usepackage{amsmath}
\usepackage{amsfonts}
\usepackage{amssymb}
\usepackage{graphicx}%
\providecommand{\U}[1]{\protect\rule{.1in}{.1in}}
\newtheorem{theorem}{Theorem}[section]

\newtheorem{corollary}[theorem]{Corollary}

\newtheorem{definition}[theorem]{Definition}

\newtheorem{lemma}[theorem]{Lemma}

\newtheorem{question}[theorem]{Question}
\newtheorem{proposition}[theorem]{Proposition}
\newtheorem{remark}[theorem]{Remark}

\newenvironment{proof}[1][Proof]{\noindent\textbf{#1.} }{\ \rule{0.5em}{0.5em}}
\begin{document}

\title{Invariant subspaces of the direct sum of forward and backward shifts on vector-valued Hardy spaces}
\author{Caixing Gu and Shuaibing Luo}
\date{}

\maketitle

\begin{abstract}
Let $S_{E}$ be the shift operator on vector-valued Hardy space $H_{E}^{2}.$
Beurling-Lax-Halmos Theorem identifies the invariant subspaces of $S_{E}$ and
hence also the invariant subspaces of the backward shift $S_{E}^{\ast}.$ In
this paper, we study the invariant subspaces of $S_{E}\oplus S_{F}^{\ast}.$ We
establish a one-to-one correspondence between the invariant subspaces of
$S_{E}\oplus S_{F}^{\ast}$ and a class of invariant subspaces of bilateral
shift $B_{E}\oplus B_{F}$ which were described by Helson and Lowdenslager
\cite{HelsonL}. As applications, we express invariant subspaces of
$S_{E}\oplus S_{F}^{\ast}$ as kernels or ranges of mixed Toeplitz operators
and Hankel operators with partial isometry-valued symbols. Our approach
greatly extends and gives different proofs of the results of C\^{a}mara and
Ross \cite{CRoss}, and Timotin \cite{Timotin} where the case with one
dimensional $E$ and $F$ was considered.

\end{abstract}

\section{Introduction}

Let $L^{2}$ be the space of square integrable functions on the unit circle
$\mathbb{T}$ with respect to the normalized Lebesgue measure. Let $H^{2}$ be
the Hardy space on the open unit disk $\mathbb{D}$. The $L^{\infty}$ and
$H^{\infty}$ are the algebras of bounded functions in $L^{2}$ and $H^{2}$
respectively. Let $E$ and $F$ be two complex separable Hilbert spaces. Let
$B(E\mathcal{)}$ be the algebra of bounded linear operators on $E$ and
$B(E,F)$ be the set of bounded linear operators from $E$ into $F.$ The
$L_{E}^{2}$ and $H_{E}^{2}$ denote $E$-valued $L^{2}$ and $H^{2}$ spaces. The
$L_{B(E)}^{\infty}$ and $H_{B(E)}^{\infty}$ are operator-valued $L^{\infty}$
and $H^{\infty}$ algebras, and $L_{B(E,F)}^{\infty}$ and $H_{B(E,F)}^{\infty}$
are operator-valued $L^{\infty}$ and $H^{\infty}$ functions.

Denote by $P$ the projection from $L_{E}^{2}$ to $H_{E}^{2}$ and $Q:=I-P$ the
projection from $L_{E}^{2}$ to $\overline{zH_{E}^{2}}:=L_{E}^{2}\ominus
H_{E}^{2},$ where $E$ is any complex separable Hilbert space. Let $\Phi\in
L_{B(E,F)}^{\infty}.$ The multiplication operator by $\Phi$ from $L_{E}^{2}$
into $L_{F}^{2}$ is denoted by $M_{\Phi}.$ The (block) Toeplitz operator
$T_{\Phi}$ from $H_{E}^{2}$ into $H_{F}^{2}$ is defined by%
\[
T_{\Phi}h =P\left[  \Phi h\right]  ,h\in H_{E}^{2}.
\]
When $E=F,$ let $T_{z}$($=T_{zI_{E}}$) denote the shift operator on $H_{E}%
^{2}$ for some $E$ which the context will make it clear. We will also use
$S_{E}$ to denote this shift operator $T_{z}$ on $H_{E}^{2}.$ Similarly
$M_{z}$($=M_{zI_{E}}$) is the bilateral shift on $L_{E}^{2}.$ We will also use
$B_{E}$ to denote this $M_{z}.$ The Toeplitz operator $T_{\Phi}=A$ is
characterized by the operator equation $T_{z}^{\ast}AT_{z}=A.$ Let $J$ be
defined on $L_{E}^{2}$ by
\[
Jf(z)=\overline{z}f(\overline{z}),f \in L_{E}^{2}.
\]
The $J$ maps $\overline{zH_{E}^{2}}$ onto $H_{E}^{2}$, and $J$ maps $H_{E}%
^{2}$ onto $\overline{zH_{E}^{2}}.$ Furthermore $J$ is a unitary operator,
\[
J^{\ast}=J,J^{2}=I,\text{ }JM_{z}^{\ast}=M_{z}J,\text{ }JQ=PJ,\text{ and
}JP=QJ.
\]
The Hankel operator $H_{\Phi}$ from $H_{E}^{2}$ into $H_{F}^{2}$ is defined by%
\[
H_{\Phi}h=JQ\left[  \Phi h\right]  =PJ\left[  \Phi h\right]  ,h \in
H_{E}^{2}.
\]
The Hankel operator $H_{\Phi}=A$ is characterized by the operator equation
$AT_{z}=T_{z}^{\ast}A$. It is easy to check that $H_{\Phi}^{\ast}%
=H_{\Phi(\overline{z})^{\ast}}.$ We emphasize again Toeplitz and Hankel
operators $T_{\Phi}$ and $H_{\Phi}$ could be between different spaces
$H_{E}^{2}$ and $H_{F}^{2}$ according to whether the symbol $\Phi\ $acts
between different spaces $E$ and $F.$

Let
\[
X=\left[
\begin{array}
[c]{cc}%
S_{E} & 0\\
0 & S_{F}^{\ast}%
\end{array}
\right]  =\left[
\begin{array}
[c]{cc}%
T_{z} & 0\\
0 & T_{z}^{\ast}%
\end{array}
\right]  :H_{E}^{2}\oplus H_{F}^{2}\rightarrow H_{E}^{2}\oplus H_{F}^{2}.
\]
We will study the invariant subspace of $X.$ The operator $X$ is a special
example of so called dual shifts of high multiplicity \cite{Gu}. When $\dim
E=\dim F=1,$ $X$ is a special example of scalar dual shifts \cite{DingY}
\cite{GuDual}. The invariant subspace problem of $X$ in scalar case was
initiated by C\^{a}mara and Ross \cite{CRoss} and stated as an open question.
The question was answered beautifully by Timotin \cite{Timotin}. Timotin's
approach uses the connection of regular factorizations of the scalar
characteristic function of $S_{E}\oplus S_{F}^{\ast}$ and the invariant
subspaces of $S_{E}\oplus S_{F}^{\ast}$ in Sz.-Nagy-Foias model theory
\cite{NFBK}. It seems difficult to extend Timotin's approach as it relies on
an explicit result on regular factorizations of a scalar characteristic
function as on Page 308 \cite{NFBK}. In this paper, we relate invariant
subspaces of $S_{E}\oplus S_{F}^{\ast}$ to invariant subspaces of the
bilateral shift $B_{E}\oplus B_{F}$ \cite{Henry}. Some examples of scalar dual
shifts are shown to be similar to the scalar bilateral shift in \cite{CRoss}.

Recall $\Theta\in H_{B(E_{1},E)}^{\infty}$ is left inner if $\Theta(z)^{\ast
}\Theta(z)=I_{E_{1}}$ a.e. $z \in\mathbb{T}$ for some $E_{1}\subset E, E_{1}
\neq\{0\}$. In addition, we say $\Theta$ is right extremal if $\Theta$ can not
be factored as $\Theta_{1} \Delta$ where $\Theta_{1} $ is left inner and
$\Delta$ is a non-constant square inner function (i.e. $\Delta\in H_{B(E_{1}%
)}^{\infty}$ is not constant valued, and $\Delta(z)^{\ast}\Delta
(z)=\Delta(z)\Delta(z)^{\ast}=I_{E_{1}}$ a.e. $z \in\mathbb{T}$). Let
$K_{\Theta} =H_{E}^{2}\ominus\Theta H_{E_{1}}^{2}$ denote the model space. The
celebrated Beurling-Lax-Halmos (BLH) Theorem \cite{Beurling} \cite{Lax}
\cite{Halmos} says an invariant subspace of $S_{E}$ is of the form $\Theta
H_{E_{1}}^{2},$ and consequently, an invariant subspace of $S_{E}^{\ast}$ is
of the form $K_{\Theta}.$ Let
\begin{equation}
\Psi(z)=\left[
\begin{array}
[c]{cc}%
C(z) & D(z)\\
A(z) & B(z)
\end{array}
\right]  :E\oplus F\rightarrow E\oplus F, \label{dephi}%
\end{equation}
where $A\in H_{B(E,F)}^{\infty},$ $B \in H_{B(F)}^{\infty},$ $C \in
L_{B(E)}^{\infty},$ $D \in L_{B(F,E)}^{\infty}.$ Let $\Lambda_{E\oplus
F}^{\infty}$ denote the set of all symbols $\Psi$ as above. Let $W_{\Psi}$ be
a bounded operator defined by
\begin{equation}
\label{wpsi}W_{\Psi}=\left[
\begin{array}
[c]{cc}%
H_{C}^{\ast} & T_{A}^{\ast}\\
H_{D}^{\ast} & T_{B}^{\ast}%
\end{array}
\right]  :H_{E}^{2}\oplus H_{F}^{2}\rightarrow H_{E}^{2}\oplus H_{F}^{2}.
\end{equation}
Note that $H_{C}$ and $H_{D}$ only depend on the anti-analytic parts $Q[C]$
and $Q[D].$ We have the following two characterizations for the invariant
subspaces of $S_{E}\oplus S_{F}^{\ast}$.

\begin{theorem}
\label{thm1} A closed subspace $N$ of $H_{E}^{2}\oplus H_{F}^{2}$ is an
invariant subspace of $S_{E}\oplus S_{F}^{\ast},$ if and only if
$N=\ker(W_{\Psi})\cap( \Theta H_{E_{1}}^{2}\oplus H_{F}^{2}) $ where $\Psi
\in\Lambda_{E\oplus F}^{\infty}$ is zero or partial isometry-valued (i.e.
$\Psi(z)^{\ast}\Psi(z)=I_{E_{0}}$ a.e. $z \in\mathbb{T}$ with $E_{0}\subset
E\oplus F$), and $\Theta\in H_{B(E_{1},E)}^{\infty}$ is zero or left inner and
right extremal.
\end{theorem}

By taking "adjoint" of the above theorem, we can use the ranges of $V_{\Phi}$
to express the invariant subspaces of $S_{E}\oplus S_{F}^{\ast}$. Let
\begin{equation}
\Phi(z) =\left[
\begin{array}
[c]{cc}%
A(z) & B(z)\\
C(z) & D(z)
\end{array}
\right]  :E\oplus F\rightarrow E\oplus F, \label{defphi1}%
\end{equation}
where $A \in H_{B(E)}^{\infty},$ $B \in H_{B(F,E)}^{\infty},$ $C \in
L_{B(E,F)}^{\infty},$ $D \in L_{B(F)}^{\infty}.$ Let $\Gamma_{E\oplus
F}^{\infty}$ denote the set of all symbols $\Phi$ as above. Let $V_{\Phi}$ be
a bounded operator defined by
\begin{equation}
V_{\Phi}=\left[
\begin{array}
[c]{cc}%
T_{A} & T_{B}\\
H_{C} & H_{D}%
\end{array}
\right]  :H_{E}^{2}\oplus H_{F}^{2}\rightarrow H_{E}^{2}\oplus H_{F}^{2}.
\label{defvfphi}%
\end{equation}
Note that $V_{\Phi}$ is essentially $W_{\Psi}^{\ast}.$

\begin{theorem}
\label{thm2} A closed subspace $N$ of $H_{E}^{2}\oplus H_{F}^{2}$ is an
invariant subspace of $S_{E}\oplus S_{F}^{*},$ if and only if $N=Span\left\{
R(V_{\Phi}),K_{\Theta}\right\}  $ (the closed linear span of $R(V_{\Phi})$ and
$K_{\Theta}$), where $\Phi\in\Gamma_{E\oplus F}^{\infty}$ is zero or partial
isometry-valued, $\Theta\in H_{B(F_{1},F)}^{\infty}$ with $F_{1}\subset F$,
$\Theta$ being zero or left inner and right extremal, and $K_{\Theta}=
H^{2}_{F} \ominus\Theta H^{2}_{F_{1}}$.
\end{theorem}

We will show in Section 2 that both $\ker(W_{\Psi})$ and $\overline{R(V_{\Phi
})}$ are invariant subspaces of $S_{E} \oplus S_{F}^{*}$.

Our approach also gives a different proof of the "scalar" case where $\dim
E=\dim F=1$ by Timotin \cite{Timotin}. In this case, $S_{E}\oplus S_{F}^{*}$
has the splitting invariant subspaces of the form $\theta_{1} H^{2}\oplus
K_{\theta_{2}}$ with $\theta_{1}, \theta_{2}$ inner functions or zero. When
$\dim E>1$ or $\dim F>1,$ it is natural to introduce the following more
general splitting invariant subspaces of $S_{E}\oplus S_{F}^{\ast}$. In order
to understand the invariant subspaces of $S_{E}\oplus S_{F}^{\ast},$ we can
focus on non-splitting invariant subspaces.

\begin{definition}
An invariant subspace $N$ of $S_{E}\oplus S_{F}^{\ast}$ is called splitting if
there exist direct decompositions $E=E_{1}\oplus E_{2},F=F_{1}\oplus F_{2},$
where both $S_{E_{1}}\oplus S_{F_{1}}^{\ast}$ and $S_{E_{2}}\oplus S_{F_{2}%
}^{\ast}$ are not zero, and a unitary operator $V\in B(E\oplus F)$ such that
\[
S_{E}\oplus S_{F}^{\ast}=V(S_{E_{1}}\oplus S_{F_{1}}^{\ast}\oplus S_{E_{2}%
}\oplus S_{F_{2}}^{\ast})V^{\ast},N=V(N_{1}\oplus N_{2}),
\]
where $N_{i}$ is an invariant subspace of $S_{E_{i}}\oplus S_{F_{i}}^{\ast}$
for $i=1,2.$
\end{definition}

For example, let $E_{1} \subsetneq E$, then the invariant subspaces of
$S_{E_{1}} \oplus S_{F}^{*}$ are splitting invariant subspaces of $S_{E}\oplus
S_{F}^{\ast}$. Also $\Theta H^{2}_{E_{0}} \oplus H^{2}_{F}$ is splitting,
where $E_{0} \subset E, \Theta\in H^{\infty}_{B(E_{0},E)}$ left inner, since
$\Theta H^{2}_{E_{0}}$ is an invariant subspace of $S_{E}\oplus0$ on
$H^{2}_{E} \oplus\{0\}$, and $H^{2}_{F}$ is an invariant subspace of $0\oplus
S_{F}^{*}$ on $\{0\} \oplus H^{2}_{F}$.

The following theorem was obtained by Timotin \cite{Timotin}, we formulate it
in a slightly different way.

\begin{theorem}
\label{Timotinthm}\cite{Timotin}Assume $\dim E=\dim F=1.$ Then $N$ is a
non-splitting invariant subspace of $S_{E}\oplus S_{F}^{\ast},$ if and only
if
\begin{equation}
N=R(V_{\Phi_{1}}),\text{ }V_{\Phi_{1}}=\left[
\begin{array}
[c]{cc}%
T_{a } & T_{b }\\
H_{c(\overline{z})} & H_{d(\overline{z})}%
\end{array}
\right]  ,\Phi_{1}(z)=\left[
\begin{array}
[c]{cc}%
a(z) & b(z)\\
c(\overline{z}) & d(\overline{z})
\end{array}
\right]  , \label{v1}%
\end{equation}
where $a,b,c,d\in H^{\infty}$, $a$ and $b$ are not proportional, and $\Phi
_{1}(z)^{\ast}\Phi_{1}(z)=I$ a.e. $z \in\mathbb{T}$.
\end{theorem}

\begin{theorem}
\label{mandonerange2} Assume $\dim E <\infty$ and $\dim F=1.$ If $N$ is a
non-splitting invariant subspace of $S_{E}\oplus S_{F}^{\ast}$, then there is
a partial isometry $V_{\Phi}$ such that $N=R(V_{\Phi})$.
\end{theorem}

When $\dim E=1$, see Theorem \ref{typeiikernel} for the description of
invariant subspaces of $S_{E}\oplus S_{F}^{\ast}$ in terms of $\ker(W_{\Psi}%
)$. Note that in Theorem \ref{Timotinthm} and Theorem \ref{mandonerange2}, no
closed linear span is needed, so it\ still take some efforts to obtain Theorem
\ref{Timotinthm} and Theorem \ref{mandonerange2} from the general Theorem
\ref{thm2}. Indeed, there are several questions await future work. For
example, how to characterize invariant subspace $N$ of $N=R(V_{\Phi})$? Can
$N$ be just the linear span of $R(V_{\Phi})$ and $K_{\Theta}$? For given two
invariant subspaces $N_{1}=\ker(W_{\Psi_{1}})\cap( \Theta_{1} H_{E_{1}}%
^{2}\oplus H_{F}^{2}) $ and $N_{2}=\ker(W_{\Psi_{2}})\cap( \Theta_{2}
H_{E_{2}}^{2}\oplus H_{F}^{2}) ,$ if $N_{1}=N_{2}$ what is the relation
between $\left\{  \Psi_{1},\Theta_{1} \right\}  $ and $\left\{  \Psi
_{1},\Theta_{2} \right\}  $? These questions and other related operator and
function theoretic questions are touched upon in this paper and deserved to be
studied further.

Now we outline the paper. In Section 2, we discuss some basic properties of
$W_{\Psi}$ and $V_{\Phi}.$ These operators are mixed Hankel and Toeplitz type
operators, so they are interesting in their own right because Hankel and
Toeplitz operators are important and have wide applications \cite{BS2}
\cite{Pe}. By using commutant lifting theorem \cite{NFBK}, we compute the
norms of $W_{\Psi}$ and $V_{\Phi}$ as distance problems which is a slight
generalization of Nehari's theorem \cite{Nehari}. We also discuss the kernels
and ranges of $W_{\Psi}$ and $V_{\Phi}$ as they are related to the invariant
subspaces of $S_{E}\oplus S_{F}^{\ast}$.

In Section 3, we recall the invariant subspace theorem for the bilateral
shifts $B_{E}$ by Helson and Lowdenslager \cite{Henry} \cite{HelsonL}
\cite{Srini1}. We then establish a one-to-one correspondence between the
invariant subspaces of $S_{E}\oplus S_{F}^{\ast}$ and a special class of
invariant subspaces of bilateral shift $B_{E}\oplus B_{F}.$ Since there are
two types of invariant subspaces of $B_{E}\oplus B_{F},$ one is simply
invariant and another is doubly invariant, they naturally lead to two types
of invariant subspaces of $S_{E}\oplus S_{F}^{\ast}$ which we call type I and
type II. In Section 4, we first express the type I invariant subspaces of
$S_{E}\oplus S_{F}^{\ast}$ as $\ker(W_{\Psi})$ for some $\Psi\in
\Lambda_{E\oplus F}^{\infty}$, then we prove Theorem \ref{thm1}. By using the
unitary equivalence of $S_{E}^{\ast}\oplus S_{F}$ and $S_{F}\oplus S_{E}%
^{\ast}$, we then extablish Theorem \ref{thm2}. In Section 5, we give a more
detailed analysis of the cases $\dim E = 1, \dim F < \infty$ and $\dim
E<\infty, \dim F =1$ which proves Theorems \ref{Timotinthm} and
\ref{mandonerange2}.

In this paper, for two vector-valued or operator-valued functions $F$ and $G,$
$F(z)=G(z)$ for $z\in\mathbb{T}$ means $F(z)=G(z)$ for a.e. $z\in\mathbb{T}$.

\section{Examples of invariant subspaces of $S_{E}\oplus S_{F}^{\ast}.$}

Let
\[
X=\left[
\begin{array}
[c]{cc}%
S_{E} & 0\\
0 & S_{F}^{\ast}%
\end{array}
\right]  ,Y=\left[
\begin{array}
[c]{cc}%
S_{E} & 0\\
0 & S_{F}%
\end{array}
\right]  :H_{E}^{2}\oplus H_{F}^{2}\rightarrow H_{E}^{2}\oplus H_{F}^{2},
\]
Recall $V_{\Phi}$ is defined by (\ref{defvfphi}). It is easy to see that
\begin{equation}
XV_{\Phi}=\left[
\begin{array}
[c]{cc}%
T_{z}T_{A } & T_{z}T_{B }\\
T_{z}^{\ast}H_{C } & T_{z}^{\ast}H_{D }%
\end{array}
\right]  =\left[
\begin{array}
[c]{cc}%
T_{A } & T_{B }\\
H_{C } & H_{D }%
\end{array}
\right]  \left[
\begin{array}
[c]{cc}%
T_{z} & 0\\
0 & T_{z}%
\end{array}
\right]  =V_{\Phi}Y. \label{fundament}%
\end{equation}
The above simple relation reveals three facts:\newline(i) $V_{\Phi}\in I(Y,X)$
where $I(Y,X)$ denote the set of bounded linear operators $W$ intertwining $Y$
and $X,$ i.e. $WY=XW.$\newline(ii) The closure of the range of $V_{\Phi},$
$\overline{R(V_{\Phi})}$ is an invariant subspace of $X.$\newline(iii)
$\ker(V_{\Phi}^{\ast})$ is an invariant subspace of $S_{E}^{\ast}\oplus
S_{F}.$\newline For (i), by using the commutant lifting theorem, Theorem
\ref{lifting}, recalled below, we show that $I(Y,X)$ consists of exactly those
$V_{\Phi}.$ For (ii) and (iii), we ask which invariant subspace of $X$ is of
the form $\overline{R(V_{\Phi})}$ for some $\Phi$, or which invariant subspace
of $S_{E}^{\ast}\oplus S_{F}$ is of the form $\ker(V_{\Phi}^{\ast})$, see also
Questions \ref{invaiantr} and \ref{invaiantk} below. For an operator-valued
$A,$ we will use short-hand notation $A\in H^{\infty}$ instead of $A \in
H_{B(E,F)}^{\infty}$ where $B(E,F)$ is clear from the context.

\begin{theorem}
\label{lifting} (Page 69 \cite{NFBK}) (i) Let $T\in B(H)$ and $T^{\prime}\in
B(H^{\prime})$ be two contractions. Let $U\in B(K)$ and $U^{\prime}\in
B(K^{\prime})$ be the minimal isometric dilations of $T$ and $T^{\prime},$
respectively. Then for any operator $W\in B(H,H^{\prime})$ satisfying
$WT=T^{\prime}W,$ there exists $W_{1}\in B(K,K^{\prime})$ such that
$W_{1}U=U^{\prime}W_{1},$ $\left\Vert W_{1}\right\Vert =\left\Vert
W\right\Vert ,$ $W=P_{H^{\prime}}W_{1}|H$ and $W(K\ominus H)\subset K^{\prime
}\ominus H^{\prime}.$\newline(ii) Let $T\in B(H)$ and $T^{\prime}\in
B(H^{\prime})$ be two contractions . Let $U\in B(K)$ and $U^{\prime}\in
B(K^{\prime})$ be the co-isometric extension of $T$ and $T^{\prime},$
respectively. Then for any operator $W\in B(H,H^{\prime})$ satisfying
$WT=T^{\prime}W,$ there exists $W_{1}\in B(K,K^{\prime})$ such that
$W_{1}U=U^{\prime}W_{1},$ $\left\Vert W_{1}\right\Vert =\left\Vert
W\right\Vert ,$ $W=W_{1}|H$.
\end{theorem}


\begin{lemma}
\label{multi}Let $A$ be a bounded linear operator from $H_{F}^{2}$ to
$L_{F}^{2}$ such that $AT_{z}=M_{z}^{\ast}A.$ Then there exists $\Phi\in
L_{B(F)}^{\infty}$ such that $A=JM_{\Phi}.$
\end{lemma}

\begin{proof}
Note that $AT_{z}=M_{z}^{\ast}A$ if and only if $\left(  JA\right)
T_{z}=JM_{z}^{\ast}A=M_{z}\left(  JA\right)  .$ Therefore, there exists
$\Phi\in L_{B(F)}^{\infty}$ such that $\left(  JA\right)  h=\Phi h$ for $h \in
H_{F}^{2}.$
\end{proof}

Now we can characterize the elements in $I(Y,X)$.

\begin{theorem}
\label{vfphic}Let $W\in I(Y,X).$ Then $W=V_{\Phi}$ for some $\Phi\in
\Gamma_{E\oplus F}^{\infty}.$ Furthermore,%
\begin{equation}
\left\Vert V_{\Phi}\right\Vert =\inf_{L_{1},L_{2}\in H^{\infty}}\left\Vert
\left[
\begin{array}
[c]{cc}%
A & B\\
C & D
\end{array}
\right]  -\left[
\begin{array}
[c]{cc}%
0 & 0\\
L_{1} & L_{2}%
\end{array}
\right]  \right\Vert _{L^{\infty}}. \label{norm}%
\end{equation}

\end{theorem}

\begin{proof}
Assume $W\in I(Y,X).$ That is, $WY=XW$. Note that $Y$ is an isometry, so the
minimal isometric dilation of $Y$ is itself. The minimal isometry dilation of
$X$ is $U:=T_{z}\oplus M_{z}^{\ast}$ defined on $H_{E}^{2}\oplus L_{F}^{2}.$
By Theorem \ref{lifting} (i), there exists $W_{1}:H_{E}^{2}\oplus H_{F}%
^{2}\rightarrow H_{E}^{2}\oplus L_{F}^{2}$ such that $W_{1}Y=UW_{1},$
$\left\Vert W_{1}\right\Vert =\left\Vert W\right\Vert ,$ $W=P_{1}W_{1}$ where
$P_{1}$ is the projection from $H_{E}^{2}\oplus L_{F}^{2}$ onto $H_{E}%
^{2}\oplus H_{F}^{2}.$ Write%
\[
W_{1}=\left[
\begin{array}
[c]{cc}%
A_{1} & A_{2}\\
A_{3} & A_{4}%
\end{array}
\right]  :H_{E}^{2}\oplus H_{F}^{2}\rightarrow H_{E}^{2}\oplus L_{F}^{2}.
\]
Then $W_{1}Y=UW_{1}$ becomes%
\[
\left[
\begin{array}
[c]{cc}%
A_{1}T_{z} & A_{2}T_{z}\\
A_{3}T_{z} & A_{4}T_{z}%
\end{array}
\right]  =\left[
\begin{array}
[c]{cc}%
T_{z}A_{1} & T_{z}A_{2}\\
M_{z}^{\ast}A_{3} & M_{z}^{\ast}A_{4}%
\end{array}
\right]  .
\]
Therefore, there exist $A , B \in H^{\infty}$ such that $A_{1}=T_{A}%
,B_{1}=T_{B}.$ By Lemma \ref{multi}, there exist $C_{1}, D_{1}\in L^{\infty}$
such that $A_{3}=JM_{C_{1}},A_{4}=JM_{D_{1}}.$ Thus%
\begin{equation}
W_{1}=\left[
\begin{array}
[c]{cc}%
I & 0\\
0 & J
\end{array}
\right]  \left[
\begin{array}
[c]{cc}%
T_{A } & T_{B }\\
M_{C_{1} } & M_{D_{1} }%
\end{array}
\right]  \text{ and }\left\Vert W_{1}\right\Vert =\left\Vert \left[
\begin{array}
[c]{cc}%
A & B\\
C_{1} & D_{1}%
\end{array}
\right]  \right\Vert _{L^{\infty}} \label{w1norm}%
\end{equation}
Now $W=P_{1}W_{1}$ implies that $W=V_{\Phi}$ where $H_{C}=H_{C_{1}}$ and
$H_{D}=H_{D_{1}}.$

Next we prove (\ref{norm}). It is clear that $\left\Vert V_{\Phi}\right\Vert
\leq$ right side of (\ref{norm}). By $\left\Vert W_{1}\right\Vert =\left\Vert
W\right\Vert $ and (\ref{w1norm}), the infimum in the right side of
(\ref{norm}) is attained with $\left[
\begin{array}
[c]{cc}%
L_{1} & L_{2}%
\end{array}
\right]  =\left[
\begin{array}
[c]{cc}%
C -C_{1} & D -D_{1}%
\end{array}
\right]  .$
\end{proof}

In the above theorem, when $A=B=C=0,$ we have Nehari's theorem \cite{Nehari}.
The following lemma follows from BLH (Beurling-Lax-Halmos) Theorem and
(\ref{fundament}).

\begin{lemma}
$\ker(V_{\Phi})$ is $T_{z}$-invariant. So there exist $E_{0}\subseteq E\oplus
F$ and left inner function $\Theta$ such that $\ker(V_{\Phi})=\Theta H_{E_{0}%
}^{2}.$
\end{lemma}

Set $H_{E\oplus F}^{2} =H_{E}^{2}\oplus H_{F}^{2}.$ Let $K_{\Theta}=H_{E\oplus
F}^{2}\ominus\Theta H_{E_{0}}^{2}$ and $P_{\Theta}$ the projection from
$H_{E\oplus F}^{2}$ onto $K_{\Theta}.$ So $V_{\Phi}$ maps $K_{\Theta}$ onto
$R(V_{\Phi})$ (the range of $V_{\Phi}$) which establishes a correspondence
between $Y$-invariant subspaces and $X$-invariant subspaces. Considering the
de Branges-Rovnyak space $\mathcal{M(}\Phi)=\mathcal{M(}V_{\Phi})$ such that%
\[
\left\Vert V_{\Phi}h\right\Vert _{\mathcal{M(}\Phi)}=\left\Vert P_{\Theta
}h\right\Vert .
\]
Then $\left\Vert XV_{\Phi}h\right\Vert _{\mathcal{M(}\Phi)}=\left\Vert
V_{\Phi}Yh\right\Vert _{\mathcal{M(}\Phi)}=\left\Vert P_{\Theta}zh\right\Vert
\leq\left\Vert P_{\Theta}h\right\Vert =\left\Vert V_{\Phi}h\right\Vert
_{\mathcal{M(}\Phi)},$ so $X$ acts as a contraction on $\mathcal{M(}\Phi),$
not as an isometry in general. The $R(V_{\Phi})$ as a de Branges-Rovnyak space
is a closed subspace of $H_{E}^{2}\oplus H_{F}^{2}$ if and only if $V_{\Phi}$
is a partial isometry. This leads to our question.

Recall that an operator $G\in B(H,K)$ is a partial isometry if $G$ maps
$\ker(G)^{\perp}$ isometrically onto $R(G).$ The operator $G$ is a partial
isometry if and only if $G^{\ast}$ is a partial isometry, if and only if
$GG^{\ast}$ (or $G^{\ast}G)$ is a projection.

\begin{question}
\label{pisop}When is $V_{\Phi}$ a partial isometry?
\end{question}

The above problem seems a difficult one, since the partial isometric
characterizations of Toeplitz operator $T_{A }$ ($A $ is not necessary
analytic) and Hankel operator $H_{D }$ are only known when $A$ and $D $ are
scalar functions in $L^{\infty}$ \cite{BrownDouglas} \cite{Pe}. The above
problem is inspired by the invariant subspace problem of $X=S_{E}\oplus
S_{F}^{\ast}.$

\begin{question}
\label{invaiantr}If $N$ is an invariant subspace of $X=S_{E}\oplus S_{F}%
^{\ast},$ can we find $V_{\Phi}$ such that $V_{\Phi}$ is a partial isometry
and $N=R(V_{\Phi})?$
\end{question}

Recall $W_{\Psi}$ is defined by (\ref{wpsi}).
It is easy to see that
\begin{equation}
W_{\Psi}X=\left[
\begin{array}
[c]{cc}%
H_{C }^{\ast}T_{z} & T_{A }^{\ast}T_{z}^{\ast}\\
H_{D }^{\ast}T_{z} & T_{B }^{\ast}T_{z}^{\ast}%
\end{array}
\right]  =\left[
\begin{array}
[c]{cc}%
T_{z}^{\ast}H_{C }^{\ast} & T_{z}^{\ast}T_{A }^{\ast}\\
T_{z}^{\ast}H_{D }^{\ast} & T_{z}^{\ast}T_{B}^{\ast}%
\end{array}
\right]  =Y^{\ast}W_{\Psi}. \label{funda2}%
\end{equation}
Thus $\ker(W_{\Psi})$ is an invariant subspace of $S_{E}\oplus S_{F}^{\ast}.$
This leads to the following related question.

\begin{question}
\label{invaiantk}If $N$ is an invariant subspace of $X=S_{E}\oplus S_{F}%
^{\ast},$ can we find $W_{\Psi}$ such that $N=\ker(W_{\Psi})?$
\end{question}

We may ask if we can choose a partial isometry $W_{\Psi}$ such that
$N=\ker(W_{\Psi}).$ We remark that in Questions \ref{invaiantr} and
\ref{invaiantk}, we assume both $S_{E}$ and $S_{F}^{\ast}$ are present. But in
fact, the case where either $S_{E}$ or $S_{F}^{\ast}$ is absent is also
interesting and not completely resolved. We will study Questions
\ref{invaiantr} and \ref{invaiantk} in detail in Section 4, and then establish
Theorems \ref{thm1} and \ref{thm2}.

Let $\Phi_{1} ,$ $\Phi_{2} \in\Gamma_{E\oplus F}^{\infty}.$ By Douglas lemma,
$R(V_{\Phi_{1}})\subset R(V_{\Phi_{2}})$ if and only if $V_{\Phi_{1} }%
V_{\Phi_{1}}^{\ast}\leq\lambda V_{\Phi_{2}}V_{\Phi_{2}}^{\ast}$ for some
$\lambda>0,$ if and only if there exists a bounded linear operator $W$ such
that $V_{\Phi_{1}}=V_{\Phi_{2}}W.$ The precise form of $W$ can be obtained by
using commutant lifting theorem.

\begin{proposition}
\label{inequality}Let $\Phi_{1} ,$ $\Phi_{2} \in\Gamma_{E\oplus F}^{\infty}.$
Then $V_{\Phi_{1}}V_{\Phi_{1}}^{\ast}\leq V_{\Phi_{2}}V_{\Phi_{2}}^{\ast}$ if
and only if there exists a contractive analytic Toeplitz operator $T_{\Omega}$
on $H_{E\oplus F}^{2}$ such that $V_{\Phi_{1}}=V_{\Phi_{2}}T_{\Omega}.$
\end{proposition}

\begin{proof}
Since $V_{\Phi_{1}}V_{\Phi_{1}}^{\ast}\leq V_{\Phi_{2}}V_{\Phi_{2}}^{\ast},$
there exists a contraction $W:H=\overline{R(V_{\Phi_{2}}^{\ast})}\rightarrow
H^{\prime}=\overline{R(V_{\Phi_{1}}^{\ast})}$ such that $WV_{\Phi_{2}}^{\ast
}=V_{\Phi_{1}}^{\ast}.$ By (\ref{fundament}), $R(V_{\Phi_{i}}^{\ast})$ is
$T_{z}^{\ast}$-invariant. Hence%
\[
\left(  WY^{\ast}\right)  V_{\Phi_{2}}^{\ast}=WV_{\Phi_{2}}^{\ast}X^{\ast
}=V_{\Phi_{1}}^{\ast}X^{\ast}=Y^{\ast}V_{\Phi_{1}}^{\ast}=Y^{\ast}WV_{\Phi
_{2}}^{\ast}.
\]
That is, $W(Y^{\ast}|H)=(Y^{\ast}|H^{\prime})W.$ By Theorem \ref{lifting}
(ii), there exists $W_{1}:H_{E\oplus F}^{2}\rightarrow H_{E\oplus F}^{2}$ such
that $W_{1}Y^{\ast}=Y^{\ast}W_{1}$ and $W=W_{1}|H.$ Therefore, $W_{1}^{\ast}$
is an analytic Toeplitz operator $T_{\Omega}$ and $V_{\Phi_{1}}=V_{\Phi_{2}%
}W^{\ast}=V_{\Phi_{2}}T_{\Omega}.$
\end{proof}

In the above proposition, when $V_{\Phi_{1}}$ and $V_{\Phi_{2}}$ just contain
Hankel operators, the result also follows from Theorem 4 in \cite{GuJFA}. The
following result for $W_{\Psi}$ can be proved similarly as Theorem
\ref{vfphic} and Proposition \ref{inequality}.

\begin{theorem}
Let $W\in I(X,Y^{\ast}).$ Then $W=W_{\Psi}$ for some $\Psi\in\Lambda_{E\oplus
F}^{\infty}.$ Furthermore,%
\[
\left\Vert W_{\Psi}\right\Vert =\inf_{L_{1},L_{2}\in H^{\infty}}\left\Vert
\left[
\begin{array}
[c]{cc}%
C & D\\
A & B
\end{array}
\right]  -\left[
\begin{array}
[c]{cc}%
L_{1} & L_{2}\\
0 & 0
\end{array}
\right]  \right\Vert _{L^{\infty}}.
\]
In addition, let $\Psi_{1} ,$ $\Psi_{2} \in\Gamma_{E\oplus F}^{\infty}.$ Then
$W_{\Psi_{1}}^{\ast}W_{\Psi_{1}}\leq W_{\Psi_{2}}^{\ast}W_{\Psi_{2}}$ if and
only if there exists a contractive analytic Toeplitz operator $T_{\Omega}$ on
$H_{E\oplus F}^{2}$ such that $W_{\Psi_{1}}=T_{\Omega}^{\ast}W_{\Psi_{2}}.$
\end{theorem}

\section{The connection of invariant subspaces of $S_{E}\oplus S_{F}^{\ast}$
and $B_{E}\oplus B_{F}$}

The obvious invariant subspaces $M$ of $X = S_{E}\oplus S_{F}^{\ast}$ are of
the form $M_{1}\oplus M_{2}$ where $M_{1}$ is an invariant subspace of
$M_{z},$ $M_{2}$ is an invariant subspace of $T_{z}^{\ast}.$ Those $M$ are
splitting invariant subspaces of $X.$ Recall an operator-valued analytic
function $\Theta\in H_{B(E,F)}^{\infty}$ is left inner if $\Theta^{\ast
}(z)\Theta(z)=I_{E}$ a.e. $z \in\mathbb{T}$. Let $K_{\Theta} =H_{F}^{2}%
\ominus\Theta H_{E}^{2}$ denote the model space. The $M_{1},M_{2}$ are
well-known by BLH Theorem. We say an invariant subspace $M$ of an operator $T$
on $H$ is a nontrivial invariant subspace if $M \neq\{0\}, H$.

\begin{theorem}
\label{BLH} (Beurling-Lax-Halmos theorem) (i) Let $M_{1}$ be a nontrivial
invariant subspace of $T_{z}$ on $H_{E}^{2}.$ Then $M_{1}=\Theta H_{E_{1}}%
^{2}=T_{\Theta}H_{E_{1}}^{2}$, where $E_{1}\subseteq E,$ $\Theta\in
H_{B(E_{1},E)}^{\infty}$ is left inner. \newline(ii) Let $M_{2}$ be a
nontrivial invariant subspace of $T_{z}^{*}$ on $H_{F}^{2}.$ Then
$M_{2}=K_{\Theta}=H_{F}^{2}\ominus\Theta H_{F_{1}}^{2},$ where $F_{1}\subseteq
F,$ $\Theta\in H_{B(F_{1},F)}^{\infty}$ is left inner.
\end{theorem}

We recall the invariant subspace theorem for the bilateral shifts $B_{E}$ by
Helson and Lowdenslager \cite{HelsonL} which is a generalization of
BLH\ Theorem and Wiener's result on doubly invariant subspaces of a scalar
bilateral shift. We first introduce the concept of range functions.

\begin{definition}
A range function $K$ of $E$ is a function on the circle $\mathbb{T}$ taking
values in the closed subspaces of $E$. $K$ is measurable if the orthogonal
projection $P(z)$ on $K(z)$ is measurable in the operator sense, that is,
$\left\langle P(z)e_{1},e_{2}\right\rangle $ is a measurable scalar function
for all $e_{1},e_{2}\in E.$ Range functions which are equal almost everywhere
are identified. For each measurable range function $K$ of $E$, let $M_{K}$ be
the set of all functions $f$ in $L_{E}^{2}$ such that $f(z)$ lies in $K(z)$
almost everywhere.
\end{definition}

We say an operator-valued function $U\in L_{B(E_{0},E)}^{\infty}$ is
isometry-valued if $E_{0} \neq\{0\}$ and $U(z)^{*}U(z) = I_{E_{0}}$ for a.e.
$z \in\mathbb{T}$, and unitary-valued if $U(z)^{*}U(z) = I_{E_{0}}$ and
$U(z)U(z)^{*}= I_{E}$ for a.e. $z \in\mathbb{T}$. In this paper, when we write
$U(z)^{*}U(z) = I_{E_{0}}$ for a.e. $z \in\mathbb{T}$, we always have $E_{0}
\neq\{0\}$. The following invariant subspace theorem is key to our study of
invariant subspaces of $X.$

\begin{theorem}
\label{Helson}(Page 61 and Page 64 \cite{Henry}) Let $M$ be an invariant
subspace of $M_{z}$ on $L_{E}^{2}.$ Then $M=UH_{E_{0}}^{2}\oplus M_{K}$ where
$E_{0}\subseteq E,$ $U\in L_{B(E_{0},E)}^{\infty}$ is zero or isometry-valued
function, and $K$ is a measurable range function of $E$ such that $K(z)\perp
U(z)E_{0}$ for a.e. $z\in\mathbb{T}$. Furthermore, $U_{1}H_{E_{0}}^{2}%
=U_{2}H_{E_{1}}^{2},$ where $E_{0}\subseteq E$, $E_{1}\subseteq E,$ $U_{1}$
and $U_{2}$ are isometry-valued functions, if and only if there exists a
constant unitary $W:E_{0}\rightarrow E_{1}$ such that $U_{1}=U_{2}W.$
\end{theorem}

The subspace $M_{K}$ is doubly invariant (equivalently reducing) in the sense
that $zM_{K}\subset M_{K}$ and $z^{-1}M_{K}=\overline{z}M_{K}\subset M_{K}.$
So $zM_{K}=M_{K}.$

To study the invariant subspaces of $X$, we make the following reductions.
Let
\[
D_{1}=\left[
\begin{array}
[c]{cc}%
M_{z} & 0\\
0 & T_{z}^{\ast}%
\end{array}
\right]  :L_{E}^{2}\oplus H_{F}^{2}\rightarrow L_{E}^{2}\oplus H_{F}^{2}.
\]
If $N$ is an invariant subspace of $X,$ $N$ is also an invariant subspace of
$D_{1}$ since for $h\in M,$ $D_{1}h=Xh\in M.$ This observation reduces the
study of an invariant subspace $N$ of $X$ to an invariant subspace $N_{1}$ of
$D_{1}$ such that $N_{1}\subset H_{E}^{2}\oplus H_{F}^{2}.$

Let
\[
D_{2}=\left[
\begin{array}
[c]{cc}%
M_{z} & 0\\
0 & T_{z}%
\end{array}
\right]  :L_{E}^{2}\oplus H_{F}^{2}\rightarrow L_{E}^{2}\oplus H_{F}^{2}.
\]
Let $J_{1}:L_{E}^{2}\rightarrow L_{E}^{2}$ be the unitary operator defined by
\[
J_{1}f(z)=f(\overline{z}),f \in L_{E}^{2}.
\]
$J_{1}^{2}=I,$ $J_{1}^{\ast}=J_{1}.$ Note that
\[
M_{z}J_{1}=J_{1}M_{z}^{\ast}\text{ on }L_{E}^{2}.
\]
Thus
\[
D_{1}J_{2}=\left[
\begin{array}
[c]{cc}%
M_{z} & 0\\
0 & T_{z}^{\ast}%
\end{array}
\right]  \left[
\begin{array}
[c]{cc}%
J_{1} & 0\\
0 & I
\end{array}
\right]  =J_{2}\left[
\begin{array}
[c]{cc}%
M_{z}^{\ast} & 0\\
0 & T_{z}^{\ast}%
\end{array}
\right]  =J_{2}D_{2}^{\ast}\text{ where }J_{2}=\left[
\begin{array}
[c]{cc}%
J_{1} & 0\\
0 & I
\end{array}
\right]  .
\]
Note that $N_{1}$ is an invariant subspace of $D_{1}$ if and only if
$N_{2}:=J_{2}N_{1}$ is an invariant subspace of $D_{2}^{\ast}.$ This
observation reduces the study of an invariant subspace $N_{1}$ of $D_{1}$ to
an invariant subspace $N_{2}$ of $D_{2}^{\ast}$ such that $N_{1}=J_{2}%
N_{2}\subset H_{E}^{2}\oplus H_{F}^{2}$ or $N_{2}\subset\overline{H_{E}^{2}%
}\oplus H_{F}^{2}.$ Equivalently, $N_{2}^{\prime}=\left(  L_{E}^{2}\oplus
H_{F}^{2}\right)  \ominus N_{2}$ is an invariant subspace of $D_{2}$ such that
$N_{2}^{\prime}\supset zH_{E}^{2}.$ Let%
\[
D_{3}=\left[
\begin{array}
[c]{cc}%
M_{z} & 0\\
0 & M_{z}%
\end{array}
\right]  :L_{E}^{2}\oplus L_{F}^{2}\rightarrow L_{E}^{2}\oplus L_{F}^{2}%
\]
Then an invariant subspace $N_{2}^{\prime}$ of $D_{2}$ is an invariant
subspace $N_{3}$ of $D_{3}$ such that $N_{3}\subset L_{E}^{2}\oplus H_{F}%
^{2}.$ We state these observations as a theorem.

\begin{theorem}
\label{main}Let $N$ be an invariant subspace of $X$. Then there exists an
invariant subspace $N_{3}$ of $M_{z}\oplus M_{z}$ on $L_{E}^{2}\oplus
L_{F}^{2}$ such that $N=J_{2}\left[  \left(  L_{E}^{2}\oplus H_{F}^{2}\right)
\ominus N_{3}\right]  = \left(  L_{E}^{2}\oplus H_{F}^{2}\right)  \ominus
J_{2}N_{3} $ where%
\begin{equation}
zH_{E}^{2}\subset N_{3}\subset L_{E}^{2}\oplus H_{F}^{2}. \label{twocond}%
\end{equation}
Conversely, if $N_{3}$ is an invariant subspace of $M_{z}\oplus M_{z}$ on
$L_{E}^{2}\oplus L_{F}^{2}$ such that (\ref{twocond}) holds, then
$J_{2}\left[  \left(  L_{E}^{2}\oplus H_{F}^{2}\right)  \ominus N_{3}\right]
$ is an invariant subspace of $X.$
\end{theorem}

The above theorem establishes a one-to-one correspondence between invariant
subspaces $N$ of $X$ and invariant subspaces $N_{3}$ of $M_{z}\oplus M_{z}$
satisfying (\ref{twocond}).

Let $P_{E}$ and $P_{F}$ be the projection from $E\oplus F$ onto $E$ and $F,$
respectively. For $U \in L_{B(E_{0},E\oplus F)}^{\infty},$ set
\[
U_{E}(z) =P_{E}U(z) ,U_{F}(z) =P_{F}U(z).
\]
Then $U_{E} \in L_{B(E_{0},E)}^{\infty}, U_{F}\in L_{B(E_{0},F)}^{\infty}$.

\begin{lemma}
\label{condition}Let $N_{3}$ be an invariant subspace of $M_{z}\oplus M_{z}$
on $L_{E}^{2}\oplus L_{F}^{2}.$ Write $N_{3}=UH_{E_{0}}^{2}\oplus M_{K}$ where
$E_{0}\subseteq E\oplus F,$ $U\in L_{B(E_{0},E\oplus F)}^{\infty}$ is an
isometry-valued function, and $K$ is a measurable range function.\newline(i)
$N_{3}\subset L_{E}^{2}\oplus H_{F}^{2}$ if and only if
\begin{equation}
M_{K}\subset L_{E}^{2},U_{F} \in H_{B(E_{0},F)}^{\infty}, K(z)\perp
U_{E}(z)E_{0}\text{ for a.e. }z\in\mathbb{T}. \label{inside}%
\end{equation}
(ii) $zH_{E}^{2}\subset N_{3}$ if and only if $zE\subset N_{3}.$\newline(iii)
If (\ref{twocond}) holds, then $K(z)\oplus U_{E}(z)E_{0}=E$ for a.e.
$z\in\mathbb{T}.$
\end{lemma}

\begin{proof}
(i) Assume $N_{3}=U H_{E_{0}}^{2}\oplus M_{K}\subset L_{E}^{2}\oplus H_{F}%
^{2}.$ Let $h\in M_{K}.$ Then $h\perp\overline{z}^{n}f$ for all $n\geq1$ and
$f\in F.$ But $\overline{z}^{m}h\in M_{K}$ for all $m\geq n.$ Hence
$\overline{z}^{m}h\perp\overline{z}^{n}f$. Equivalently, $h\perp z^{m-n}f.$
Thus $h\perp L_{F}^{2}$ and $M_{K}\subset L_{E}^{2}.$ Now $K(z)\perp
U(z)E_{0}$ becomes $K(z)\perp P_{E}U(z)E_{0}.$ Furthermore, for $g\in
H^{2}_{E_{0}},$ $f\in F$ and $n\geq1,$
\[
0=\left\langle U g,\overline{z}^{n}f\right\rangle _{L_{E}^{2}\oplus L_{F}^{2}%
}=\int\left\langle P_{F}U(z)g(z) ,\overline{z}^{n}f(z)\right\rangle _{F}%
\frac{|dz|}{2\pi}=\left\langle U_{F}g ,\overline{z}^{n}f\right\rangle
_{L_{F}^{2}}.
\]
This implies that $U_{F} \in H_{B(E_{0},F)}^{\infty}$. Conversely, if
(\ref{inside}) holds, one can verify that $N_{3}\subset L_{E}^{2}\oplus
H_{F}^{2}.$

(ii) It is clear that $zH_{E}^{2}\subset N_{3}$ implies $zE\subset N_{3}.$
Conversely, assume $zE\subset N_{3}.$ Since $N_{3}$ is invariant under the
multiplication by $z,$ $z^{n+1}E\subset z^{n}N_{3}\subset N_{3}$ for all
$n\geq1.$ Thus $zH_{E}^{2}\subset N_{3}.$

(iii) For $e\in E,$ there exists $h\in H_{E_{0}}^{2}$ such that $ze-P_{E}%
U(z)h(z)\in K(z)$ for $z\in\mathbb{T}.$ Hence by (i) $K(z)\oplus U_{E}(z)E_{0}
= K(z)\oplus P_{E}U(z)E_{0} =E.$
\end{proof}

Let $N_{3}=U H_{E_{0}}^{2}\oplus M_{K}$ be such that (\ref{twocond}) holds. If
$U =0$, then $N_{3}=M_{K}\supset zH_{E}^{2}$ simply implies that $N_{3}%
=L_{E}^{2}.$
We make the following definition.

\begin{definition}
\label{typeiandii} Let $N$ be an invariant subspace of $X.$ We say $N$ is type
I if the corresponding $N_{3}=UH_{E_{0}}^{2}.$ We say $N$ is type II if the
corresponding $N_{3}=UH_{E_{0}}^{2}\oplus M_{K}$ where $K \neq0$.
\end{definition}

First we study the type I invariant subspace of $X.$ If $N_{3}=U H_{E_{0}}%
^{2}$ is such that $U $ is unitary-valued, then we have the following result.
Recall that for an operator-valued $A,$ we will use short-hand notation $A\in
H^{\infty}$ instead of $A\in H_{B(E,F)}^{\infty}$ where $B(E,F)$ is clear from
the context.

\begin{proposition}
\label{full}Assume $N_{3}=UH_{E_{0}}^{2}$ is such that $U \in L^{\infty
}_{B(E_{0},E\oplus F)}$ is unitary-valued, $E_{0} \subset E \oplus F$. Write%
\[
U(z)=\left[
\begin{array}
[c]{cc}%
U_{E}(z)\\
U_{F}(z)
\end{array}
\right]  :E_{0}\rightarrow E\oplus F
\]
Then $zH_{E}^{2}\subset N_{3}\subset L_{E}^{2}\oplus H_{F}^{2}$ if and only if
$U_{F}, zU_{E}^{*} \in H^{\infty}$.
\end{proposition}

\begin{proof}
By (\ref{inside}), $N_{3}\subset L_{E}^{2}\oplus H_{F}^{2}$ if and only if
$U_{F} \in H^{\infty}$. Now $zH^{2}_{E} \subset N_{3}$ if and only if $\forall
f \in H^{2}_{E}, zf \in UH^{2}_{E_{0}}$, if and only if $zU_{E}^{*} f =
U^{*}(zf \oplus0) \in H^{2}_{E_{0}}$. Thus $zH_{E}^{2}\subset N_{3}$ if and
only if $zU_{E}^{*} \in H^{\infty}$.
\end{proof}

For the more general case $N_{3}=U H_{E_{0}}^{2}$ where $U$ is
isometry-valued, we first introduce the following assumption. Note that in
this case, if $zH_{E}^{2}\subset N_{3} = U H_{E_{0}}^{2}$, then by the same
argument as above, we necessarily have $zU_{E}^{*} \in H^{\infty}$, where
$U_{E}(z) =P_{E}U(z)$.

\begin{definition}
Let $U \in L_{B(E_{0},E\oplus F)}^{\infty}$ be isometry-valued with
$E_{0}\subset E\oplus F.$ Set $E_{0}^{\perp}=\left(  E\oplus F\right)  \ominus
E_{0}.$ We say $U $ is complementary isometry-valued if there exists $V \in
L_{B(E_{0}^{\perp},E\oplus F)}^{\infty}$ such that $\left[
\begin{array}
[c]{cc}%
U & V
\end{array}
\right]  $ is unitary-valued.
\end{definition}

When $E$ and $F$ are finite dimensional, if $U \in L_{B(E_{0},E\oplus
F)}^{\infty}$ is isometry-valued, then it is not hard to see that $U $ is
complementary isometry-valued. In this case, by applying a constant unitary
operator $W$ as in Theorem \ref{Helson}, if necessary we may assume
$E_{0}\subset E$ if $\dim E_{0}\leq\dim E$ and $E_{0}=E\oplus F_{1}$ where
$F_{1}\subset F$ if $\dim E_{0}\geq\dim E.$

For a matrix-valued function $A ,$ $\text{rank} A(z)$ is in general not a
constant function. But if either $A $ or $A^{\ast}\in H^{\infty},$ then
$\text{rank} A(z)$ is a constant function. This is probably known. We state
this as a lemma and include a proof.

\begin{lemma}
\label{fixed}Let $A $ be a (finite) matrix-valued function. If either $A $ or
$A^{\ast}\in H^{\infty},$ then $\text{rank} A(z)$ is a constant function for
a.e. $z \in\mathbb{T}$. If $\text{rank} A(z)=m,$ then there are $m$ fixed
columns (or $m$ fixed rows) of $A(z)$ that are linearly independent for a.e.
$z\in\mathbb{T}.$ Here fixed columns means independent of $z.$
\end{lemma}

\begin{proof}
Let $A\in H^{\infty}$ be a (finite) matrix-valued function. Let $k$ be the
maximum such that $\text{rank}A(z)=k$ for $z$ in a subset $\tau$ of
$\mathbb{T}$ with the measure of $\tau$ not zero. Then there exists a
submatrix $A_{k}(z)$ (by choosing $k$ rows and $k$ columns of $A(z)$
independent of $z$) such that $\det A_{k}(z)\neq0$ for $z\in\tau.$ But $\det
A_{k}\in H^{\infty},$ so $\det A_{k}(z)\neq0$ for a.e. $z\in\mathbb{T}$, and
hence $\text{rank}A(z)$ is a constant function on the circle.
\end{proof}

\begin{proposition}
\label{ab1}Let $N_{3}=U H_{E_{0}}^{2}$ be such that $E_{0}\subset E\oplus F$
and $U $ is complementary isometry-valued. Write $E_{0}^{\perp}=\left(
E\oplus F\right)  \ominus E_{0},$
\[
\left[
\begin{array}
[c]{cc}%
U(z) & V(z)
\end{array}
\right]  =\left[
\begin{array}
[c]{cc}%
U_{E}(z) & V_{E}(z)\\
U_{F}(z) & V_{F}(z)
\end{array}
\right]  :E_{0}\oplus E_{0}^{\perp}\rightarrow E\oplus F.
\]
Then the following two statements hold. \newline(i) $zH_{E}^{2}\subset
N_{3}\subset L_{E}^{2}\oplus H_{F}^{2}$ if and only if $U_{F} ,zU_{E}^{\ast
}\in H^{\infty}$ and there exists $V $ with $V_{E} =0$ such that $\left[
\begin{array}
[c]{cc}%
U & V
\end{array}
\right]  $ is unitary-valued. \newline(ii) When $E$ and $F$ are finite
dimensional, $zH_{E}^{2}\subset N_{3}\subset L_{E}^{2}\oplus H_{F}^{2}$ if and
only if $U_{F} ,zU_{E} ^{\ast}\in H^{\infty}$ and $\text{rank}U_{F}(z)=\dim
E_{0}-\dim E$.
\end{proposition}

\begin{proof}
(i) By (\ref{inside}), $N_{3}\subset L_{E}^{2}\oplus H_{F}^{2}$ if and only if
$U_{F}\in H^{\infty}.$ If $zH_{E}^{2}\subset N_{3}$, then $\forall e \in E$,
there exists $g \in H_{E_{0}}^{2}$ such that $Ug = ze$. Note that $\left[
\begin{array}
[c]{cc}%
U & V
\end{array}
\right]  $ is unitary-valued, the equation%
\[
U g =\left[
\begin{array}
[c]{cc}%
U & V
\end{array}
\right]  (g \oplus0)=ze
\]
is the same as
\[
g \oplus0=\left[
\begin{array}
[c]{cc}%
U & V
\end{array}
\right]  ^{*} (ze\oplus0) = zU_{E} ^{\ast}e\oplus zV_{E} ^{\ast}e,
\]
So $zU_{E}^{\ast}\in H^{\infty}$ and $V_{E} = 0$. Conversely, if $U_{F}
,zU_{E}^{\ast}\in H^{\infty}$ and there exists $V $ with $V_{E} =0$ such that
$\left[
\begin{array}
[c]{cc}%
U & V
\end{array}
\right]  $ is unitary-valued, then for $f \in H^{2}_{E}$, let $g = zU_{E}
^{\ast} f$, we have $g \in H^{2}_{E_{0}}$ and by the above discussion $Ug=zf$.
Thus $zH_{E}^{2}\subset N_{3}$.

(ii) Assume $E$ and $F$ are finite dimensional. Note that $zH_{E}^{2}\subset U
H_{E_{0}}^{2}$ implies that $\dim E_{0}\geq\dim E.$ Assume $U_{F}
,zU_{E}^{\ast}\in H^{\infty}$ and there exists $V $ with $V_{E} =0$ such that
$\left[
\begin{array}
[c]{cc}%
U & V
\end{array}
\right]  $ is unitary-valued. Then $\text{rank}U_{E}(z) = \dim E$,
\begin{align*}
\text{rank} U_{F}(z) = \text{rank}U(z) - \text{rank}U_{E}(z) = \dim E_{0} -
\dim E.
\end{align*}

On the other hand, assume $n_{1} =rankU_{F}(z)=\dim E_{0}-\dim E.$ By Lemma
\ref{fixed}, there exists $n_{1}$ columns of $U_{F}(z)$ which is a basis of
the column space of $U_{F}(z).$ Denote these $n_{1}$ columns by $\left\{
c_{1}(z),\cdots,c_{n_{1}}(z)\right\}  .$ Extend $\left\{  c_{1}(z),\cdots
,c_{n_{1}}(z)\right\}  $ to $\left\{  c_{1}(z),\cdots,c_{n_{1}}(z),d_{1}%
(z),\cdots,d_{n_{2}}(z)\right\}  $ so that
\begin{align*}
Span\left\{  c_{1}(z),\cdots,c_{n_{1}}(z),d_{1}(z),\cdots,d_{n_{2}%
}(z)\right\}   &  =F,\\
Span\left\{  c_{1}(z),\cdots,c_{n_{1}}(z)\right\}   &  \perp Span\left\{
d_{1}(z),\cdots,d_{n_{2}}(z)\right\}  ,
\end{align*}
and $\left\{  d_{1}(z),\cdots,d_{n_{2}}(z)\right\}  $ are orthonormal vectors.
Let $V_{F}(z)$ be the matrix whose columns are $\left\{  d_{1}(z),\cdots
,d_{n_{2}}(z)\right\}  .$ Then $\left[
\begin{array}
[c]{cc}%
U & V
\end{array}
\right]  $ is unitary-valued with $V_{E} =0.$ This proves (ii) and (iii) are equivalent.
\end{proof}

Now we look at the type II invariant subspaces of $X$. When $N_{3}=UH_{E_{0}%
}^{2}\oplus M_{K},$ where $E_{0}\subseteq E\oplus F,$ $U\in L_{B(E_{0},E\oplus
F)}^{\infty}$ is isometry-valued and $K$ is a nonzero measurable range
function of $E$, it is more difficult to see when $N_{3}\supset zH_{E}^{2}.$
In the case $E$ and $F$ are finite dimensional, a precise answer of when
$N_{3}\supset zH_{E}^{2}$ is obtained in Proposition \ref{type2condition}
below. First in the case $\dim E<\infty$, we have the following result.

\begin{lemma}
\label{costant}Assume $\dim E=m<\infty.$ Let $N_{3}=UH_{E_{0}}^{2}\oplus
M_{K}$ where $U\in L_{B(E_{0},E\oplus F)}^{\infty}$ is isometry-valued and $K$
is a nonzero measurable range function of $E$, $E_{0} \subset E \oplus F$. If
$zH_{E}^{2}\subset N_{3}\subset L_{E}^{2}\oplus H_{F}^{2}$, then $N_{3}
\supset(L^{2}_{E} \ominus\overline{\Theta H^{2}_{E_{2}}})$ for some $E_{2}
\subset E$ with $\dim E_{2} < m$, and $\Theta\in H^{\infty}_{B(E_{2},E)}$ zero
or left inner.
\end{lemma}

\begin{proof}
Let $m_{1}$ be maximal value of $\dim K(z).$ Then $1\leq m_{1}\leq m$, and
there is $\tau\subset\mathbb{T}$ with $|\tau|>0$ ($|\tau|$ is the measure of
$\tau$) such that $\dim K(z)=m_{1},z\in\tau$. Let $K_{1}$ be the range
function defined by $K_{1}(z)=K(z)$ if $z\in\tau$ and $K_{1}(z)=\{0\}$ if
$z\in\mathbb{T}\backslash\tau.$ Then $M_{K_{1}}\subset M_{K}.$ By the
discussion before Theorem \ref{main}, $N_{3}\supset zH_{E}^{2}$ if and only if
$J_{2}\left[  \left(  L_{E}^{2}\oplus H_{F}^{2}\right)  \ominus N_{3}\right]
\subset H_{E}^{2}\oplus H_{F}^{2},$ if and only if $(L_{E}^{2}\oplus H_{F}%
^{2})\ominus N_{3}\subset\overline{H_{E}^{2}}\oplus H_{F}^{2}.$ Note that%
\begin{align*}
\left(  L_{E}^{2}\oplus H_{F}^{2}\right)  \ominus N_{3}  &  \subset\left(
L_{E}^{2}\oplus H_{F}^{2}\right)  \ominus M_{K}\\
&  \subset\left(  L_{E}^{2}\oplus H_{F}^{2}\right)  \ominus M_{K_{1}}=\left(
L_{E}^{2}\ominus M_{K_{1}}\right)  \oplus H_{F}^{2}.
\end{align*}
Let $\left\{  v_{1}(z),\cdots,v_{m_{1}}(z)\right\}  $ be an orthogonal\ set of
unit vectors such that $\left\{  v_{1}(z)\chi_{\tau}(z),\cdots,v_{m_{1}%
}(z)\chi_{\tau}(z)\right\}  $ is an orthonormal basis of $K_{1}(z)$ for
$z\in\tau.$ Extend $\left\{  v_{1}(z),\cdots,v_{m_{1}}(z)\right\}  $ to
$\left\{  v_{1}(z),\cdots,v_{m}(z)\right\}  $ to be an orthonormal basis of
$E.$ Let $K_{2}$ be the range function defined by%
\[
K_{2}(z)=Span\left\{  v_{1}(z)(1-\chi_{\tau}(z)),\cdots,v_{m_{1}}%
(z)(1-\chi_{\tau}(z)),v_{m_{1}+1}(z),\cdots,v_{m}(z)\right\}  .
\]
So $L_{E}^{2}\ominus M_{K_{1}}=M_{K_{2}}$, and
\begin{align}
\label{containmentrel}\left(  L_{E}^{2}\oplus H_{F}^{2}\right)  \ominus
N_{3}\subset\left(  \overline{ H_{E}^{2}}\oplus H_{F}^{2}\right)  \cap\left[
M_{K_{2}}\oplus H_{F}^{2}\right] .
\end{align}
Note that $\overline{H_{E}^{2}}\cap M_{K_{2}}=\overline{H_{E}^{2}\cap
M_{K_{2}(\overline{z})}}$. If $H_{E}^{2}\cap M_{K_{2}(\overline{z})}=\{0\}$,
then we are done. Now suppose $H_{E}^{2}\cap M_{K_{2}(\overline{z})}\neq
\{0\}$. Since $H_{E}^{2}\cap M_{K_{2}(\overline{z})}$ is an invariant subspace
of $T_{z}$ on $H_{E}^{2}$, by BLH Theorem we have $H_{E}^{2}\cap
M_{K_{2}(\overline{z})}=\Theta H_{E_{2}}^{2}$ for some $E_{2}\subset E$ and
$\Theta$ left inner. Then $\Theta(z)E_{2}\subset K_{2}(\overline{z})$ for
$z\in\mathbb{T}$. So
\[
\dim E_{2}=\dim\Theta(z)E_{2}\leq\dim K_{2}(\overline{z})\leq m-m_{1},z\in
\tau.
\]
Thus (\ref{containmentrel}) implies $N_{3}\supset(L_{E}^{2}\ominus
\overline{\Theta H_{E_{2}}^{2}})$ with $\dim E_{2}<m$.
\end{proof}

If $\dim E = 1$, then the $\Theta$ in the above lemma is zero, and we obtain
the following corollary.

\begin{corollary}
\label{costant1}Assume $\dim E=1.$ Let $N_{3}=U H_{E_{0}}^{2}\oplus M_{K}$
where $U\in L_{B(E_{0},E\oplus F)}^{\infty}$ is zero or isometry-valued and
$K$ is a nonzero measurable range function of $E$, $E_{0} \subset E \oplus F$.
If $zH_{E}^{2}\subset N_{3}\subset L_{E}^{2}\oplus H_{F}^{2}$, then
$N_{3}\supset L_{E}^{2}.$ Consequently, the corresponding $N$ in Theorem
\ref{main} is equal to $\{0\} \oplus H^{2}_{F}$ or $\{0\} \oplus(H^{2}_{F}
\ominus\Lambda H^{2}_{F_{1}})$, where $F_{1} \subset F, \Lambda\in H^{\infty
}_{B(F_{1},F)}$ is left inner.
\end{corollary}

If we also assume $\dim E_{0} < \infty$ in Lemma \ref{costant}, then we have
the following characterization of $M_{K}$.

\begin{lemma}
\label{typeiifinite} Assume $\dim E=m<\infty.$ Let $N_{3}=UH_{E_{0}}^{2}\oplus
M_{K}$ where $U\in L_{B(E_{0},E\oplus F)}^{\infty}$ is isometry-valued and $K$
is a nonzero measurable range function of $E$, $E_{0} \subset E \oplus F$.
Assume also $\dim E_{0} = n < \infty$. If $zH_{E}^{2}\subset N_{3}\subset
L_{E}^{2}\oplus H_{F}^{2}$, then there are $E_{1}\subset E\oplus F, E_{1}
\perp E_{0}$, $\dim E_{1} + \dim E_{0} \geq\dim E$, $\Omega\in L_{B(E_{1}%
,E)}^{\infty}$ isometry-valued such that $M_{K} = \Omega L^{2}_{E_{1}}$.
\end{lemma}

\begin{proof}
Let $m_{1}$ be the maximal value of $\dim K(z).$ Then $1\leq m_{1}\leq m$, and
there is $\tau\subset\mathbb{T}$ with $|\tau|>0$ such that $\dim
K(z)=m_{1},z\in\tau$. Recall that $U_{E}(z) =P_{E}U(z) ,U_{F}(z) =P_{F}U(z)$.
By Lemma \ref{condition} (iii), $\text{rank}U_{E}(z)=m-m_{1},z\in\tau$. So
$\text{rank}U_{F}(z)=n-(m-m_{1})=n+m_{1}-m,z\in\tau$. Then $n+m_{1}\geq m$.
Since $U_{F}\in H^{\infty}$, by Lemma \ref{fixed}, we have $\text{rank}%
U_{F}(z)=n+m_{1}-m,z\in\mathbb{T}$. Then $\text{rank}U_{E}(z)=m-m_{1}%
,z\in\mathbb{T}$ and $\dim K(z)=m_{1},z\in\mathbb{T}$. Let $\left\{
v_{1}(z),\cdots,v_{m_{1}}(z)\right\}  $ be an orthonormal basis of $K(z)$ and
$\Omega(z)$ be the matrix-valued function whose columns are $\left\{
v_{1}(z),\cdots,v_{m_{1}}(z)\right\}  .$ Then $\Omega$ is isometry-valued with
$\Omega(z)^{\ast}\Omega(z)=I_{E_{1}}$ a.e. $z\in\mathbb{T}$ and $M_{K}=\Omega
L_{E_{1}}^{2}$ with $\dim E_{1}=m_{1}$. By applying a constant unitary $W$ as
in Theorem \ref{Helson}, we can assume $E_{1}\subset E\oplus F,E_{1}\perp
E_{0}$. The proof is complete.
\end{proof}

When $E$ and $F$ are finite dimensional, the following result gives a complete
characterization of type II invariant subspace $N$ of $S_{E}\oplus S_{F}%
^{\ast}$ in terms of $N_{3}.$

\begin{proposition}
\label{type2condition} Assume $E$ and $F$ are finite dimensional. Then the
following hold. \newline(i) Let $N$ be a type II invariant subspace of
$S_{E}\oplus S_{F}^{\ast}$. Then $N=J_{2}\left[  \left(  L_{E}^{2}\oplus
H_{F}^{2}\right)  \ominus N_{3}\right]  $ where $N_{3}=UH_{E_{0}}^{2}%
\oplus\Omega L_{E_{1}}^{2},$ $E_{0},E_{1}\subset E\oplus F$, $E_{1}\perp
E_{0},$ $\dim E_{1}+\dim E_{0}\geq\dim E$, $U\in L_{B(E_{0},E\oplus
F)}^{\infty}$ is zero or isometry-valued, and $\Omega\in L_{B(E_{1}%
,E)}^{\infty}$ is isometry-valued. \newline(ii) Let $N_{3}=U H_{E_{0}}%
^{2}\oplus\Omega L_{E_{1}}^{2}$ where $E_{0}, E_{1} \subset E \oplus F$,
$E_{1} \perp E_{0}$ and $U\in L_{B(E_{0},E\oplus F)}^{\infty}$, $\Omega\in
L_{B(E_{1},E)}^{\infty}$ are isometry-valued. Write
\[
U(z)=\left[
\begin{array}
[c]{c}%
U_{E}(z)\\
U_{F}(z)
\end{array}
\right]  :E_{0}\rightarrow E\oplus F.
\]
Then $zH_{E}^{2}\subset N_{3}\subset L_{E}^{2}\oplus H_{F}^{2}$ if and only if
$U_{F} ,zU_{E} ^{\ast}\in H^{\infty}$ and $\text{rank} U_{F}(z)=\dim
E_{0}+\dim E_{1}-\dim E.$
\end{proposition}

\begin{proof}
Part (i) follows from Lemma \ref{typeiifinite}. Now we prove part (ii). Write%
\[
\Omega(z)=\left[
\begin{array}
[c]{c}%
\Omega_{E}(z)\\
0
\end{array}
\right]  ,V(z)=\left[
\begin{array}
[c]{c}%
V_{E}(z)\\
V_{F}(z)
\end{array}
\right]
\]
By Lemma \ref{condition}, $\left[
\begin{array}
[c]{cc}%
U & \Omega
\end{array}
\right]  $ is isometry-valued. Thus there exists $V\in L_{B((E_{0}\oplus
E_{1})^{\perp},E\oplus F)}^{\infty}$ such that $\left[
\begin{array}
[c]{ccc}%
U & \Omega & V
\end{array}
\right]  $ is unitary-valued. Assume $zH_{E}^{2}\subset N_{3}\subset L_{E}%
^{2}\oplus H_{F}^{2}$. For $e\in E,$ $g\in H_{E_{0}}^{2},f\in L_{E_{1}}%
^{2},0\in L_{(E_{0}\oplus E_{1})^{\perp}}^{2},$ the equation%
\[
\left[
\begin{array}
[c]{ccc}%
U & \Omega & V
\end{array}
\right]  (g\oplus f\oplus0)=ze = ze\oplus0
\]
is the same as
\[
g\oplus f\oplus0=zU_{E}^{\ast}e\oplus z\Omega_{E}^{\ast}e\oplus zV_{E}^{\ast
}e.
\]
Thus $zU_{E}^{\ast}\in H^{\infty}$ and $zV_{E}^{\ast}=0.$ Note that there
exists $V$ with $V_{E}=0$ such that $\left[
\begin{array}
[c]{ccc}%
U & \Omega & V
\end{array}
\right]  $ is unitary-valued if and only if
\[
\text{rank}V_{F}(z)=\dim E+\dim F-\dim E_{0}-\dim E_{1},
\]
if and only if $\text{rank}U_{F}(z)=\dim F-\text{rank}V_{F}(z)=\dim E_{0}+\dim
E_{1}-\dim E.$ For the converse part, we apply the same argument as in
Proposition \ref{ab1} to obtain the conclusion.
\end{proof}

\section{Invariant subspaces of $S_{E}\oplus S_{F}^{\ast}$ as kernels of
$W_{\Psi}$ or ranges of $V_{\Phi}.$}

The following lemma presents a condition for $W_{\Psi}$ and $V_{\Phi}$ to be
partial isometries.

\begin{lemma}
\label{partial1}Assume $U \in L_{B(E_{0},E\oplus F)}^{\infty}$ is
isometry-valued and
\[
U(z)=\left[
\begin{array}
[c]{c}%
U_{E}(z)\\
U_{F}(z)
\end{array}
\right]  =\left[
\begin{array}
[c]{c}%
A(z)\\
C(z)
\end{array}
\right]  :E_{0}\rightarrow E\oplus F
\]
where $C ,zA ^{\ast}\in H^{\infty}.$ The following three statements
hold.\newline(i) If $A(z)A(z)^{\ast}=I_{E}$ a.e. $z \in\mathbb{T}$, then
$W_{\Psi}$ is a partial isometry, where
\[
W_{\Psi} =\left[
\begin{array}
[c]{cc}%
H_{\overline{z}A}^{\ast} & T_{C}^{\ast}%
\end{array}
\right]  ,\Psi(z)=\left[
\begin{array}
[c]{c}%
\overline{z}A(z)\\
C(z)
\end{array}
\right]  .
\]
(ii) If $C(z)C(z)^{\ast}=I_{F}$ a.e. $z \in\mathbb{T}$, then $V_{\Phi}$ is a
partial isometry, where%
\[
V_{\Phi}=\left[
\begin{array}
[c]{c}%
T_{zA(\overline{z})}\\
H_{C(\overline{z})}%
\end{array}
\right]  ,\Phi(z)=\left[
\begin{array}
[c]{c}%
zA(\overline{z})\\
C(\overline{z})
\end{array}
\right]  .
\]
(iii) If $U$ is unitary-valued, then $W_{\Psi}^{*} W_{\Psi}+ V_{\Phi}V_{\Phi
}^{*} = I_{H^{2}_{E\oplus F}}$ and $\ker(W_{\Psi}) = R(V_{\Phi})$.
\end{lemma}

\begin{proof}
(i) Note that $A(z)A(z)^{\ast}=I_{E}$ a.e. $z \in\mathbb{T}$ just says
$zA^{\ast}$ is left inner. Thus $T_{zA^{\ast}}$ is a partial isometry. Now%
\begin{align*}
W_{\Psi}W_{\Psi}^{\ast}  &  =\left[
\begin{array}
[c]{cc}%
H_{\overline{z}A }^{\ast} & T_{C }^{\ast}%
\end{array}
\right]  \left[
\begin{array}
[c]{c}%
H_{\overline{z}A }\\
T_{C }%
\end{array}
\right]  =H_{\overline{z}A }^{\ast}H_{\overline{z}A }+T_{C }^{\ast}T_{C }\\
&  =T_{zA ^{\ast}\overline{z}A }-T_{\overline{z}A }^{\ast}T_{\overline{z}A
}+T_{C ^{\ast}C }=I_{H^{2}_{E_{0}}}-T_{\overline{z}A }^{\ast}T_{\overline{z}A
},
\end{align*}
where in the last equality we used that $A(z)^{*}A(z)+C(z)^{*}C(z) =
U(z)^{*}U(z) = I_{E_{0}}$ a.e. $z \in\mathbb{T}$. Therefore, $W_{\Psi}$ is a
partial isometry and $W_{\Psi}$ maps $\ker(W_{\Psi})^{\perp}$ isometrically
onto the model space $K_{zA^{\ast}}=H_{E_{0}}^{2}\ominus zA^{\ast}H_{E}^{2}.$

(ii) Note that $C(z)C(z)^{\ast}=I_{F}$ a.e. $z \in\mathbb{T}$ just says
$C(\overline{z})^{\ast}$ is left inner. Thus $T_{C(\overline{z})^{\ast}}$ is a
partial isometry. Now%
\begin{align*}
V_{\Phi}^{\ast}V_{\Phi}  &  =\left[
\begin{array}
[c]{cc}%
T_{zA(\overline{z})}^{\ast} & H_{C(\overline{z})}^{\ast}%
\end{array}
\right]  \left[
\begin{array}
[c]{c}%
T_{zA(\overline{z})}\\
H_{C(\overline{z})}%
\end{array}
\right]  =T_{A(\overline{z}) ^{\ast}A(\overline{z}) }+H_{C(\overline{z}%
)}^{\ast}H_{C(\overline{z})}\\
&  =T_{A(\overline{z})^{\ast}A(\overline{z})}+T_{C(\overline{z})^{\ast
}C(\overline{z})}-T_{C(\overline{z})}^{\ast}T_{C(\overline{z})}=I_{H^{2}%
_{E_{0}}}-T_{C(\overline{z})^{\ast}}T_{C(\overline{z})^{\ast}}^{\ast}.
\end{align*}
Therefore, $V_{\Phi}$ is a partial isometry and $V_{\Phi}^{*}$ maps
$\ker(V_{\Phi}^{*}) ^{\perp}$ isometrically model space $K_{C(\overline
{z})^{\ast}}=H_{E_{0}}^{2}\ominus C(\overline{z})^{\ast}H_{F}^{2}.$

(iii) If $U$ is unitary-valued, then for a.e. $z \in\mathbb{T}$
\[
A(z)A(z)^{\ast}=I_{E}, C(z)C(z)^{\ast}=I_{F}, A(z)C(z)^{*} = 0.
\]
So $W_{\Psi}$ and $V_{\Phi}$ are partial isometries. By a straightforward
computation as in (i) and (ii), we have $W_{\Psi}^{*} W_{\Psi}+ V_{\Phi
}V_{\Phi}^{*} = I_{H^{2}_{E\oplus F}}$. It then follows that $\ker(W_{\Psi}) =
R(V_{\Phi})$.
\end{proof}

\begin{remark}
\label{wpsivphi} In the above lemma, $W_{\Psi}$ and $V_{\Phi}$ are a little
different from the expressions (\ref{wpsi}) and (\ref{defvfphi}) in the
introduction. In this paper, we are in the situation that $E_{0} \subset E
\oplus F$, so the above $W_{\Psi}$ and $V_{\Phi}$ are operators from
$H^{2}_{E} \oplus H^{2}_{F}$ to $H^{2}_{E} \oplus H^{2}_{F}$. By applying a
constant unitary $W$ as in Theorem \ref{Helson}, we may assume either $E_{0}
\subset E$ or $E \subset E_{0}$. Now let $E_{0}^{\perp}= (E\oplus F) \ominus
E_{0}$,
\[
W_{\widetilde{\Psi}}=\left[
\begin{array}
[c]{cc}%
H_{\overline{z}A}^{\ast} & T_{C}^{\ast}\\
0 & 0
\end{array}
\right]  , \widetilde{\Psi}(z) =\left[
\begin{array}
[c]{cc}%
\overline{z}A(z) & 0\\
C(z) & 0
\end{array}
\right]  :E_{0}\oplus E_{0}^{\perp}\rightarrow E\oplus F.
\]
Note that when $E_{0} \subset E$, $E_{0} \oplus E_{0}^{\perp}= [E_{0}
\oplus(E\ominus E_{0})]\oplus F$, and when $E_{0} \supset E$, $E_{0} \oplus
E_{0}^{\perp}= E \oplus[(E_{0} \ominus E) \oplus F_{1}]$ with $F_{1} = F
\ominus(E_{0} \ominus E)$. Thus $W_{\widetilde{\Psi}}$ has the form
(\ref{wpsi}), and $W_{\widetilde{\Psi}}$ and $W_{\Psi}$ have the same action
on $H^{2}_{E} \oplus H^{2}_{F}$. The same can be said about $V_{\Phi}$.
\end{remark}

\begin{definition}
Assume $\Phi\in L^{\infty}_{B(E)}$. $\Phi$ is called partial isometry-valued
if $\Phi(z)^{*}\Phi(z) = I_{E_{0}}$ a.e. $z \in\mathbb{T}$ for some $E_{0}
\subset E$.
\end{definition}

In Lemma \ref{partial1}, $\Psi\in L^{\infty}_{B(E_{0},E\oplus F)}$ is
isometry-valued with $\Psi(z)^{\ast}\Psi(z)=I_{E_{0}}$ a.e. $z \in\mathbb{T}$.
Since $\widetilde{\Psi}(z)^{*} \widetilde{\Psi}(z) = \Psi(z)^{*} \Psi(z)
\oplus0 = I_{E_{0}} \oplus0$ a.e. $z \in\mathbb{T}$, we have $\widetilde{\Psi}
\in L^{\infty}_{B(E\oplus F)}$ is partial isometry-valued. In the following,
we will not differentiate $W_{\Psi}$ and $W_{\widetilde{\Psi}}$, $V_{\Phi}$
and $V_{\widetilde{\Phi}}$. The only difference is whether the symbol is
isometry-valued or partial isometry-valued, which should be clear from the context.

\begin{theorem}
\label{type1k}Let $N_{3}=UH_{E_{0}}^{2}$ be such that $zH_{E}^{2}\subset
N_{3}\subset L_{E}^{2}\oplus H_{F}^{2}$, where $U \in L_{B(E_{0},E\oplus
F)}^{\infty}$ is isometry-valued, $E_{0} \subset E \oplus F$. Write%
\[
U(z)=\left[
\begin{array}
[c]{c}%
U_{E}(z)\\
U_{F}(z)
\end{array}
\right]  =\left[
\begin{array}
[c]{c}%
A(z)\\
C(z)
\end{array}
\right]  :E_{0}\rightarrow E\oplus F \label{u00}%
\]
Let $N=\left(  L_{E}^{2}\oplus H_{F}^{2}\right)  \ominus J_{2}N_{3}$. Then $N
=\ker(W_{\Psi})$, where%
\begin{align}
\label{wpsiandpsi}W_{\Psi}=\left[
\begin{array}
[c]{cc}%
H_{\overline{z}A }^{\ast} & T_{C }^{\ast}%
\end{array}
\right]  ,\Psi(z)=\left[
\begin{array}
[c]{c}%
\overline{z}A(z)\\
C(z)
\end{array}
\right]  .
\end{align}
Furthermore, if $U $ is complementary isometry-valued, then $W_{\Psi}$ is a
partial isometry.
\end{theorem}

\begin{proof}
By Lemma \ref{condition} and the remark after Proposition \ref{full}, $C,
zA^{*}\in H^{\infty}.$ Note that $N \subset H^{2}_{E} \oplus H^{2}_{F}$. Let
$f\in H_{E}^{2}$ and $g\in H_{F}^{2},$ then $f \oplus g \perp J_{2}N_{3}$,
i.e. $f(\overline{z})\oplus g \perp UH_{E_{0}}^{2}$ if and only if $U ^{\ast
}\left[  f(\overline{z})\oplus g \right]  \perp H_{E_{0}}^{2},$ if and only if
($P$ is the projection to the analytic part of a vector-valued $L^{2}$
function)
\begin{align*}
0  &  =PU ^{\ast}\left[
\begin{array}
[c]{c}%
f(\overline{z})\\
g
\end{array}
\right]  =P\left[  A^{\ast} f(\overline{z})+C ^{\ast}g\right] \\
&  =H_{\overline{z}A }^{\ast}f +T_{C }^{\ast}g =W_{\Psi}\left[
\begin{array}
[c]{c}%
f\\
g
\end{array}
\right]  ,
\end{align*}
where in the last equality, we used that $P\left[  A^{\ast} f(\overline
{z})\right]  =PJ\left[  \overline{z}A^{\ast}(\overline{z})f \right]
=H_{\overline{z}A^{\ast}(\overline{z})}f=H_{\overline{z}A }^{\ast}f.$ Thus
$N=\ker(W_{\Psi}).$

If $U $ is complementary isometry-valued, then by Proposition \ref{ab1}, there
exists $V_{2} $ such that
\[
\left[
\begin{array}
[c]{cc}%
U & V
\end{array}
\right]  =\left[
\begin{array}
[c]{cc}%
A & 0\\
C & V_{2}%
\end{array}
\right]  \text{ is unitary-valued.}%
\]
Hence $\left[
\begin{array}
[c]{cc}%
U(z) & V(z)
\end{array}
\right]  \left[
\begin{array}
[c]{cc}%
U(z) & V(z)
\end{array}
\right]  ^{\ast}=I_{E\oplus F}$ implies that $A(z)A(z)^{\ast}=I_{E}.$ By Lemma
\ref{partial1}, $W_{\Psi}$ is a partial isometry.
\end{proof}

Let $\Psi\in L^{\infty}_{B(E_{0},E\oplus F)}$ be defined by (\ref{wpsiandpsi})
with $C \in H^{\infty}$. Then it is not hard to verify that $\ker(W_{\Psi})$
is an invariant subspace of $S_{E}\oplus S_{F}^{\ast}$, see also (\ref{funda2}).

\begin{corollary}
\label{maink1} Assume $E$ and $F$ are finite dimensional. A closed subspace
$N$ of $H_{E}^{2}\oplus H_{F}^{2}$ is a type I invariant subspace of
$S_{E}\oplus S_{F}^{\ast},$ if and only if $N=\ker(W_{\Psi})$ where $\Psi$ is
isometry-valued $\Psi(z)^{\ast}\Psi(z)=I_{E_{0}}$ a.e. $z\in\mathbb{T}$ with
$E_{0}\subset E\oplus F,$ $W_{\Psi}$ is a partial isometry, and $\Psi$ can be
written as
\[
\Psi(z)=\left[
\begin{array}
[c]{c}%
\overline{z}A(z)\\
C(z)
\end{array}
\right]  :E_{0}\rightarrow E\oplus F
\]
with $zA^{\ast},C\in H^{\infty}$ and $\text{rank}C(z)=\dim E_{0}-\dim E$.
Furthermore, the representation of $N=\ker(W_{\Psi})$ is unique in the
following sense: if there exists another $\Psi_{1}$ such that $N=\ker
(W_{\Psi_{1}})$ where $\Psi_{1}$ is isometry-valued $\Psi_{1}(z)^{\ast}%
\Psi_{1}(z)=I_{E_{1}}$ a.e. $z\in\mathbb{T}$ with $E_{1}\subset E\oplus F,$
and $\Psi_{1}$ can be written as
\[
\Psi_{1}(z)=\left[
\begin{array}
[c]{c}%
\overline{z}A_{1}(z)\\
C_{1}(z)
\end{array}
\right]  :E_{1}\rightarrow E\oplus F
\]
with $zA_{1}^{\ast},C_{1}\in H^{\infty}$ and $\text{rank}C_{1}(z)=\dim
E_{1}-\dim E$, then there exists a constant unitary $W:E_{0}\rightarrow E_{1}$
such that $\Psi=\Psi_{1}W.$
\end{corollary}

\begin{proof}
The first part follows from Proposition \ref{ab1} and Theorem \ref{type1k}.
For uniqueness, by the proof of above theorem, $N=\left(  L_{E}^{2}\oplus
H_{F}^{2}\right)  \ominus J_{2}N_{3}$ where%
\[
N_{3}=UH_{E_{0}}^{2}=U_{1}H_{E_{1}}^{2},U(z)=\left[
\begin{array}
[c]{c}%
A(z)\\
C(z)
\end{array}
\right]  ,U_{1}(z)=\left[
\begin{array}
[c]{c}%
A_{1}(z)\\
C_{1}(z)
\end{array}
\right]  .
\]
By the uniqueness part of Theorem \ref{Helson}, there exists a constant
unitary $W:E_{0}\rightarrow E_{1}$ such that $U=U_{1}W.$ Consequently,
$\Psi=\Psi_{1}W.$
\end{proof}

Let $\Psi\in L^{\infty}_{B(E_{0},E\oplus F)}$ be defined by (\ref{wpsiandpsi})
with $C \in H^{\infty}$. If we don't put any other conditions on $\Psi$, then
the uniqueness question for $\Psi$ with $N = \ker(W_{\Psi})$ seems
complicated. See also Theorem \ref{typeiikernel}. We state the following question.

\begin{question}
\label{kernelunique} Let $N=\ker(W_{\Psi_{i}})$ be invariant subspaces of
$S_{E}\oplus S_{F}^{\ast}$, where $\Psi_{i}\in L_{B(E_{i},E\oplus F)}^{\infty
}$ is isometry-valued for $i=1,2$, what is the relationship between $\Psi_{1}$
and $\Psi_{2}?$
\end{question}

Next we discuss type II invariant subspace of $S_{E}\oplus S_{F}^{\ast}.$ We
first prove a lemma. Recall from the introduction that a left inner function
$\Theta$ is right extremal if $\Theta$ can not be factored as $\Theta
_{1}\Delta$ where $\Theta_{1}$ is left inner and $\Delta$ is a non-constant
square inner function. The concept of right extremal inner function first
appeared in \cite{GuDOParkj}, where the kernels of Hankel operators and their
connection to invariant subspaces of $S_{E}$ were studied. See also \cite{CHL}
where the invariant subspaces of $S_{E}$ of infinite multiplicity is discussed.

\begin{lemma}
\label{rightextremal}Let $M_{K}\subset L_{E}^{2}$ where $K$ is a range
function of $E.$ Let $H_{1}=\left\{  h \in H_{E}^{2}:h \perp M_{K}\right\}  .$
Then $H_{1}$ is an invariant subspace of $S_{E}$. If $H_{1} \neq\{0\}$, then
there exists a left inner, right extremal function $\Theta$ such that
$H_{1}=\Theta H_{E_{1}}^{2}$ for some $E_{1} \subset E$.
\end{lemma}

\begin{proof}
Note that $H_{1} = H^{2}_{E} \cap M_{K}^{\perp}$, so it is closed. If $h \in
H_{1}, h \perp M_{K},$ then $S_{E}h \perp S_{E}M_{K}$ since $S_{E}$ is an
isometry. But $S_{E}M_{K}=M_{K}.$ Thus $z h\perp M_{K}$ and $zh\in H_{1}$. In
conclusion, $H_{1}$ is an invariant subspace of $S_{E}.$ By BLH Theorem, if
$H_{1} \neq\{0\}$, then $H_{1}=\Theta H_{E_{1}}^{2}$ for some $E_{1} \subset
E$ and $\Theta$ left inner. Note that $H_{1} \oplus M_{K}$ is an invariant
subspace of $M_{z}$ on $L^{2}_{E}$, Theorem \ref{Helson} then implies that for
any $f \in M_{K}$, $\Theta(z)^{*}f(z)=0$ a.e. $z \in\mathbb{T}$, also see Page
165 \cite{AlemanMalman}. If $\Theta$ is not right extremal, then
$\Theta=\Theta_{1}\Delta$, where $\Theta_{1} $ is left inner and $\Delta$ is a
non-constant square inner function. So for any $f \in M_{K}$, $\Theta
_{1}(z)^{*} f(z) = \Delta(z)\Theta(z)^{*}f(z)=0$ a.e. $z \in\mathbb{T}$. This
clearly implies that $\Theta_{1}H^{2}_{E_{1}} \perp M_{K}$ and $\Theta
_{1}H^{2}_{E_{1}} \subset\Theta H^{2}_{E_{1}}$. However, this contradicts the
fact that $\Delta$ is a non-constant square inner function. Therefore,
$\Theta$ is right extremal.
\end{proof}

\begin{theorem}
\label{type2k}Let $N_{3}=U H_{E_{0}}^{2}\oplus M_{K}$ be such that $zH_{E}%
^{2}\subset N_{3}\subset L_{E}^{2}\oplus H_{F}^{2},$ where $U \in
L_{B(E_{0},E\oplus F)}^{\infty}$ is isometry-valued and $K$ is a range
function of $E$, $E_{0} \subset E \oplus F$. Write
\[
U(z)=\left[
\begin{array}
[c]{c}%
U_{E}(z)\\
U_{F}(z)
\end{array}
\right]  =\left[
\begin{array}
[c]{c}%
zA(\overline{z})\\
C(z)
\end{array}
\right]  . \label{u1}%
\]
Let $N=\left(  L_{E}^{2}\oplus H_{F}^{2}\right)  \ominus J_{2}N_{3}$. Then
$N=\ker(W_{\Psi})\cap( \Theta H_{E_{1}}^{2}\oplus H_{F}^{2}) $, where%
\begin{equation}
W_{\Psi}=\left[
\begin{array}
[c]{cc}%
H_{A(\overline{z})}^{\ast} & T_{C}^{\ast}%
\end{array}
\right]  :H_{E}^{2}\oplus H_{F}^{2}\rightarrow H_{E_{0}}^{2},\Psi(z)=\left[
\begin{array}
[c]{c}%
A(\overline{z})\\
C(z)
\end{array}
\right]  , \label{u2}%
\end{equation}
and $\Theta$ is zero or left inner, right extremal with $\Theta(z)^{\ast
}\Theta(z)=I_{E_{1}}$ a.e. $z \in\mathbb{T}$ and $E_{1}\subset E.$
Furthermore, if $A \in H^{\infty},$ then $A =\Theta\Gamma$ where $\Gamma\in
H_{B(E_{0},E_{1})}^{\infty},$ and
\[
N=\left\{  \left[
\begin{array}
[c]{c}%
\Theta f\\
g
\end{array}
\right]  \in H_{E}^{2}\oplus H_{F}^{2}:\left[
\begin{array}
[c]{c}%
f\\
g
\end{array}
\right]  \in\ker(W_{\Psi_{1}})\subset H_{E_{1}}^{2}\oplus H_{F}^{2}\right\}
\]
with $\Psi_{1} = \left[
\begin{array}
[c]{cc}%
\Gamma(\overline{z}) & C
\end{array}
\right]  ^{tr} \in L^{\infty}_{B(E_{0},E_{1} \oplus F)}$ isometry-valued.
\end{theorem}

\begin{proof}
Since $N_{3}=U H_{E_{0}}^{2}\oplus M_{K},$ let $f \in H_{E}^{2}$ and $g \in
H_{F}^{2},$ then $f \oplus g \perp J_{2}N_{3}$ if and only if (i)
$f(\overline{z})\perp M_{K }$ (equivalently $f $ $\perp M_{K(\overline{z})}$)
and (ii) $f(\overline{z})\oplus g \perp U H_{E_{0}}^{2}$.

For (i), $f $ $\perp M_{K(\overline{z})}$. Set $H_{1} =\left\{  h \in
H_{E}^{2}:h \perp M_{K(\overline{z})}\right\}  .$ By previous lemma,
$H_{1}\subset H_{E}^{2}$ is an invariant subspace of $S_{E}$, and when $H_{1}
\neq\{0\}$, there exists a left inner, right extremal function $\Theta$ such
that $H_{1}=\Theta H_{E_{1}}^{2}$ for some $E_{1}\subset E$ and $\Theta
(z)^{\ast}\Theta(z)=I_{E_{1}} $ a.e. $z \in\mathbb{T}$.

For (ii), let $H_{2} =\left\{  f \oplus g \in H_{E}^{2}\oplus H_{F}%
^{2}:f(\overline{z})\oplus g \perp U H_{E_{0}}^{2}\right\}  $. By the proof of
Theorem \ref{type1k}, $H_{2}=\ker(W_{\Psi}),$ where $W_{\Psi}$ is defined in
(\ref{u2}). In conclusion, $N=\ker(W_{\Psi})\cap( \Theta H_{E_{1}}^{2}\oplus
H_{F}^{2})$.

Assume $A \in H^{\infty}.$ By Lemma \ref{condition}, $C \in H^{\infty}$ and
$K(z) \oplus A(\overline{z})E_{0} = E$ a.e. $z \in\mathbb{T}$. Then
$A(z)E_{0}\perp K(\overline{z})$ a.e. $z \in\mathbb{T}$. So for any $e_{0}\in
E_{0},A e_{0}\in H_{1}$. If $H_{1} = \{0\}$, then $A = 0 = \Theta$ and the
conclusion holds. If $H_{1} \neq\{0\}$, then $H_{1}=\Theta H_{E_{1}}^{2}$, and
$A =\Theta\Gamma$ where $\Gamma\in H_{B(E_{0},E_{1})}^{\infty}$. A prior, we
only have $\Gamma e_{0}\in H_{E_{1}}^{2}$ for each $e_{0}\in E_{0},$ then
$\Gamma=\Theta^{\ast}A $ tell us that $\Gamma\in H_{B(E_{0},E_{1})}^{\infty}.$
Note that $U(z)^{\ast}U(z)=I_{E_{0}}$ implies
\begin{equation}
I_{E_{0}}=A(\overline{z})^{\ast}A(\overline{z})+C(z)^{\ast}C(z)=\Gamma
(\overline{z})^{\ast}\Gamma(\overline{z})+C(z)^{\ast}C(z). \label{isome}%
\end{equation}
Let $f \in H_{E_{1}}^{2}$ and $g \in H_{F}^{2}.$ Then $\Theta f\oplus g\in N$
if and only if
\begin{align*}
0  &  =W_{\Psi}( \Theta f\oplus g) =H_{A(\overline{z})}^{\ast}\Theta f+T_{C
}^{\ast}g=H_{A ^{\ast}}\Theta f+T_{C }^{\ast}g\\
&  =H_{\Gamma^{\ast}}f+T_{C }^{\ast}g=W_{\Psi_{1}}( f\oplus g)
\end{align*}
where
\begin{align*}
\label{psionetypeii}W_{\Psi_{1}}=\left[
\begin{array}
[c]{cc}%
H_{\Gamma(\overline{z})}^{\ast} & T_{C }^{\ast}%
\end{array}
\right]  :H_{E_{1}}^{2}\oplus H_{F}^{2}\rightarrow H_{E_{0}}^{2},\Psi
_{1}(z)=\left[
\begin{array}
[c]{c}%
\Gamma(\overline{z})\\
C(z)
\end{array}
\right]  ,
\end{align*}
and by (\ref{isome}) $\Psi_{1} $ is isometry-valued.
\end{proof}

\begin{remark}
When $E$ and $F$ are finite dimensional, if $N$ is a type II invariant
subspace in the above theorem, then by Proposition \ref{type2condition},
$M_{K}=\Omega L_{E_{2}}^{2}$ where $\Omega\in L_{B(E_{2},E)}^{\infty}$ is
isometry-valued. Also $A\in H^{\infty}$ ,$\text{rank} C(z)=\dim E_{0}+\dim
E_{2}-\dim E,$ and $\Theta\in H^{\infty}_{B(E_{1},E)}$ is zero or left inner
satisfying $\Omega^{\ast}(\overline{z}) \Theta=0.$
\end{remark}

In the above theorem if $K=0$, then $\Theta(z)=I_{E},E_{1}=E$ and Theorem
\ref{type2k} reduces to Theorem \ref{type1k}. If $N_{3}=M_{K}$ is such that
$zH_{E}^{2}\subset N_{3}\subset L_{E}^{2}\oplus H_{F}^{2},$ then as we have
noted in the remark above Definition \ref{typeiandii} that $N_{3}=L_{E}^{2}$.
So the corresponding $N=H_{F}^{2}$. If we let $\Psi=0,\Theta=0$, then
$N=\ker(W_{\Psi})\cap(\Theta H_{E_{1}}^{2}\oplus H_{F}^{2})$.

Now we are ready to prove the main Theorem \ref{thm1}. We restate this theorem here.

\begin{theorem}
\label{maink}A closed subspace $N$ of $H_{E}^{2}\oplus H_{F}^{2}$ is an
invariant subspace of $S_{E}\oplus S_{F}^{\ast},$ if and only if
$N=\ker(W_{\Psi})\cap( \Theta H_{E_{1}}^{2}\oplus H_{F}^{2}) $ where $\Psi
\in\Lambda_{E\oplus F}^{\infty}$ is zero or partial isometry-valued with
$\Psi(z)^{\ast}\Psi(z)=I_{E_{0}}$ a.e. $z \in\mathbb{T}$ and $E_{0}\subset
E\oplus F,$ and $\Theta$ is zero or left inner, right extremal with
$E_{1}\subset E.$
\end{theorem}

\begin{proof}
One direction is clear since both $\ker(W_{\Psi})$ and $\Theta H_{E_{1}}%
^{2}\oplus H_{F}^{2}$ are invariant subspace of $S_{E}\oplus S_{F}^{\ast}.$
The other direction follows from Theorem \ref{type2k} and the above discussion.
\end{proof}

Note that if $\Theta_{i} \in H_{B(E_{i},E)}^{\infty}$ with $E_{i}\subset E,$
$i=1,2$ are two left inner functions, then by BLH Theorem $\Theta_{1}H_{E_{1}%
}^{2}\cap\Theta_{2}H_{E_{2}}^{2}=\Theta H_{E_{3}}^{2},$ where $E_{3} \subset
E, \Theta\in H_{B(E_{3},E)}^{\infty}$ is zero or left inner. If $N=\ker
(W_{\Psi_{i}})\cap( \Theta_{i} H_{E_{i}}^{2}\oplus H_{F}^{2}) $ for $i=1,2,$
then $N=\ker(W_{\Psi_{i}})\cap( \Theta H_{E_{3}}^{2}\oplus H_{F}^{2}) $.
Therefore, for two different representations of $N,$ we may assume $\Theta_{1}
=\Theta_{2}.$ This leads to the following uniqueness question for invariant
subspaces of $S_{E}\oplus S_{F}^{\ast}.$ When $\Theta(z) = I_{E}, E_{3} = E$,
the following question reduces to Question \ref{kernelunique}.

\begin{question}
\label{mgeneral}Let $N=\ker(W_{\Psi_{i}})\cap( \Theta H_{E_{3}}^{2}\oplus
H_{F}^{2}) $ where $\Psi_{i}\in L_{B(E_{i},E\oplus F)}^{\infty}$ is
isometry-valued, $E_{i} \subset E\oplus F$ for $i=1,2$, $E_{3} \subset E$ and
$\Theta$ is zero or left inner, what is the relationship between $\Psi_{1}$
and $\Psi_{2}?$
\end{question}

Next we discuss how to express the invariant subspaces of $S_{E}\oplus
S_{F}^{\ast}$ using ranges of $V_{\Phi}.$ We first discuss type I invariant
subspaces where $U$ is unitary-valued.

\begin{theorem}
\label{mainr}Assume $N_{3}=UH_{E_{0}}^{2}$ is such that $zH_{E}^{2}\subset
N_{3}\subset L_{E}^{2}\oplus H_{F}^{2},$ $E_{0}\subset E\oplus F$ and $U \in
L_{B(E_{0},E\oplus F)}^{\infty}$ is unitary-valued. Write%
\[
U(z)=\left[
\begin{array}
[c]{c}%
U_{E}(z)\\
U_{F}(z)
\end{array}
\right]  =\left[
\begin{array}
[c]{c}%
zA(\overline{z})\\
C(z)
\end{array}
\right]  .
\]
Let $N=\left(  L_{E}^{2}\oplus H_{F}^{2}\right)  \ominus J_{2}N_{3}$. Then $N
=R(V_{\Phi})=\ker(W_{\Psi}),$ where%
\begin{align*}
V_{\Phi}  &  =\left[
\begin{array}
[c]{cc}%
T_{A }\\
H_{C(\overline{z})}
\end{array}
\right]  ,\Phi(z)=\left[
\begin{array}
[c]{cc}%
A(z) \\
C(\overline{z})
\end{array}
\right]  ,\\
W_{\Psi}  &  =\left[
\begin{array}
[c]{cc}%
H_{A(\overline{z})}^{\ast} & T_{C }^{\ast}%
\end{array}
\right]  ,\Psi(z)=\left[
\begin{array}
[c]{cc}%
A(\overline{z}) \\
C(z)
\end{array}
\right]  =\Phi(\overline{z}).
\end{align*}
Furthermore, $W_{\Psi}$ and $V_{\Phi}$ are partial isometries. If $\Phi_{1} $
is another unitary-valued function such that
\[
\Phi_{1}(z)=\left[
\begin{array}
[c]{cc}%
A_{1}(z)\\
C_{1}(\overline{z})
\end{array}
\right]
\]
with $A_{1} ,C_{1}\in H^{\infty}$ and $R(V_{\Phi_{1}})=R(V_{\Phi}),$ then
there exists a constant unitary $W\in B(E_{0})$ such that $\Phi_{1}=\Phi W.$
\end{theorem}

\begin{proof}
By Proposition \ref{full}, we have $A, C \in H^{\infty}$. Then by Lemma
\ref{partial1} and Theorem \ref{type1k}, we have $N =R(V_{\Phi})=\ker(W_{\Psi
}),$ and $W_{\Psi}$ and $V_{\Phi}$ are partial isometries.

The formula $\Phi_{1} =\Phi W$ is the same as $\Psi_{1} =\Phi_{1}(\overline
{z})=\Phi(\overline{z})W=\Psi_{1} W$ which follows from the same argument as
in the proof of Corollary \ref{maink1}.
\end{proof}

In the more general case $N_{3}=U H_{E_{0}}^{2},$ where $U $ is not
unitary-valued, we have the following result. Note that by the remark after
Proposition \ref{full}, if $zH_{E}^{2}\subset N_{3} = U H_{E_{0}}^{2}$, then
$zU_{E}^{*} \in H^{\infty}$.

\begin{proposition}
Let $N_{3}=UH_{E_{0}}^{2}$ be such that $zH_{E}^{2}\subset N_{3}\subset
L_{E}^{2}\oplus H_{F}^{2}$, where $U \in L_{B(E_{0},E\oplus F)}^{\infty}$ is
isometry-valued, $E_{0} \subset E \oplus F$. Write%
\begin{equation}
U(z)=\left[
\begin{array}
[c]{c}%
U_{E}(z)\\
U_{F}(z)
\end{array}
\right]  =\left[
\begin{array}
[c]{c}%
zA(\overline{z})\\
C(z)
\end{array}
\right]  :E_{0}\rightarrow E\oplus F. \label{u3}%
\end{equation}
Let $N=\left(  L_{E}^{2}\oplus H_{F}^{2}\right)  \ominus J_{2}N_{3}$. Then
$N\supset R(V_{\Phi}),$ where%
\[
V_{\Phi}=\left[
\begin{array}
[c]{c}%
T_{A}\\
H_{C(\overline{z})}%
\end{array}
\right]  ,\Phi(z)=\left[
\begin{array}
[c]{c}%
A(z)\\
C(\overline{z})
\end{array}
\right]  .
\]
Furthermore, if $U$ is complementary isometry-valued, then by Proposition
\ref{ab1}, there exists $\Omega$ such that $\left[
\begin{array}
[c]{cc}%
U & \Omega
\end{array}
\right]  $ is unitary-valued and
\begin{equation}
\left[
\begin{array}
[c]{cc}%
U(z) & \Omega(z)
\end{array}
\right]  =\left[
\begin{array}
[c]{cc}%
zA(\overline{z}) & 0\\
C(z) & \Omega_{F}(z)
\end{array}
\right]  :E_{0}\oplus E_{0}^{\perp}\rightarrow E\oplus F, \label{u4}%
\end{equation}
where $A,C\in H^{\infty}.$ In this case, $N\supset R(V_{\widehat{\Phi}}),$
where%
\[
V_{\widehat{\Phi}}=\left[
\begin{array}
[c]{cc}%
T_{A} & 0\\
H_{C(\overline{z})} & H_{\Omega_{F}(\overline{z})}%
\end{array}
\right]  ,\widehat{\Phi}(z)=\left[
\begin{array}
[c]{cc}%
A(z) & 0\\
C(\overline{z}) & \Omega_{F}(\overline{z})
\end{array}
\right]  .
\]

\end{proposition}

\begin{proof}
If $U(z)$ has the form (\ref{u3}), then $C \in H^{\infty}$, and $zU_{E}^{*}
\in H^{\infty}$. So $A, C \in H^{\infty}$.

We prove the case $\left[
\begin{array}
[c]{cc}%
U & \Omega
\end{array}
\right]  ,$ the case without $\Omega$ is a modification by setting $\Omega=0.$
Let $h_{1}\in H_{E_{0}}^{2},h_{2}\in H_{E_{0}^{\perp}}^{2}.$ Set
\[
\left[
\begin{array}
[c]{c}%
f\\
g
\end{array}
\right]  =V_{\widehat{\Phi}}\left[
\begin{array}
[c]{c}%
h_{1}\\
h_{2}%
\end{array}
\right]  =\left[
\begin{array}
[c]{cc}%
T_{A} & 0\\
H_{C(\overline{z})} & H_{\Omega_{F}(\overline{z})}%
\end{array}
\right]  \left[
\begin{array}
[c]{c}%
h_{1}\\
h_{2}%
\end{array}
\right]  .
\]
Since $A\in H^{\infty},$
\[
\left[
\begin{array}
[c]{c}%
f\\
g
\end{array}
\right]  =\left[
\begin{array}
[c]{c}%
P\left[  Ah_{1}\right] \\
PJ\left[  C(\overline{z})h_{1}+\Omega_{F}(\overline{z})h_{2}\right]
\end{array}
\right]  =\left[
\begin{array}
[c]{c}%
Ah_{1}\\
P(CJh_{1}+\Omega_{F}Jh_{2})
\end{array}
\right]  .
\]
Therefore, for $k\in H_{E_{0}}^{2},$ $0\in H_{E_{0}^{\perp}}^{2}$%
\begin{align*}
&  \left\langle \left[
\begin{array}
[c]{c}%
f\\
g
\end{array}
\right]  ,J_{2}Uk\right\rangle =\left\langle \left[
\begin{array}
[c]{c}%
f\\
g
\end{array}
\right]  ,J_{2}\left[
\begin{array}
[c]{cc}%
U & \Omega
\end{array}
\right]  \left[
\begin{array}
[c]{c}%
k\\
0
\end{array}
\right]  \right\rangle \\
&  =\left\langle f(\overline{z}),zA(\overline{z})k\right\rangle +\left\langle
g,Ck\right\rangle \\
&  =\left\langle A(\overline{z})h_{1}(\overline{z}),zA(\overline
{z})k\right\rangle +\left\langle CJh_{1}+\Omega_{F}Jh_{2},Ck\right\rangle \\
&  =\left\langle \left[
\begin{array}
[c]{cc}%
U & \Omega
\end{array}
\right]  \left[
\begin{array}
[c]{c}%
Jh_{1}\\
Jh_{2}%
\end{array}
\right]  ,\left[
\begin{array}
[c]{cc}%
U & \Omega
\end{array}
\right]  \left[
\begin{array}
[c]{c}%
k\\
0
\end{array}
\right]  \right\rangle \\
&  =\left\langle \left[
\begin{array}
[c]{c}%
Jh_{1}\\
Jh_{2}%
\end{array}
\right]  ,\left[
\begin{array}
[c]{c}%
k\\
0
\end{array}
\right]  \right\rangle =0,
\end{align*}
where the third equality follows from $C\in H^{\infty},$ the second to the
last equality follows from the assumption that $\left[
\begin{array}
[c]{cc}%
U & \Omega
\end{array}
\right]  $ is unitary-valued.
\end{proof}

Let $\Psi(z) = \Phi(\overline{z})$, $W_{\Psi} =\left[
\begin{array}
[c]{cc}%
H_{A(\overline{z})}^{\ast} & T_{C }^{\ast}%
\end{array}
\right] $. Then by Theorem \ref{type1k}, $N = \ker(W_{\Psi})$. One can also
verify directly that $W_{\Psi} V_{\Phi} = 0$ and $W_{\Psi}V_{\widehat{\Phi}} =
0$.

We remark that by Proposition \ref{ab1}, when $E$ and $F$ are finite
dimensional, there exists $\Omega$ such that $\left[
\begin{array}
[c]{cc}%
U & \Omega
\end{array}
\right]  $ is unitary-valued and (\ref{u4}) holds. In fact, in this case there
exist infinitely many such $\Omega$ since $\Omega_{F}$ can be replaced by
$\Omega_{F}\Theta$ where $\Theta\in L^{\infty}_{B(E_{0}^{\perp})}$ is any
unitary-valued functions. By the above proposition, it seems difficult to
express a type I invariant subspace $N$ of $S_{E}\oplus S_{F}^{\ast}$ as
$R(V_{\Phi})$ since $N\supset R(V_{\widehat{\Phi}})$ for infinitely many
$V_{\widehat{\Phi}}.$ However, by taking "adjoint" of Theorem \ref{maink}, we
have the following result.

\begin{lemma}
\label{mainkr}A closed subspace $M$ of $H_{E}^{2}\oplus H_{F}^{2}$ is an
invariant subspace of $S_{E}^{\ast}\oplus S_{F},$ if and only if $M$ is
$Span\left\{  R(W_{\Psi}^{\ast}),K_{\Theta}\right\}  $ (the closed linear span
of $R(W_{\Psi}^{\ast})$ and $K_{\Theta}$) where $\Psi\in\Lambda_{E\oplus
F}^{\infty}$ is zero or partial isometry-valued, $\Theta\in H^{\infty
}_{B(E_{1},E)}$ is zero or left inner, right extremal, and $K_{\Theta}=
H^{2}_{E} \ominus\Theta H^{2}_{E_{1}}$.
\end{lemma}

\begin{proof}
Note that $M$ is an invariant subspace of $S_{E}^{\ast}\oplus S_{F}$ if and
only if $N:=\left(  H_{E}^{2}\oplus H_{F}^{2}\right)  \ominus M$ is an
invariant subspace of $S_{E}\oplus S_{F}^{\ast}.$ By Theorem \ref{maink},
$N=\ker(W_{\Psi})\cap( \Theta H_{E_{1}}^{2}\oplus H_{F}^{2}) .$ Therefore,%
\[
M=Span\left\{  \ker(W_{\Psi}) ^{\perp},( \Theta H_{E_{1}}^{2}\oplus H_{F}^{2})
^{\perp}\right\}  =Span\left\{  R(W_{\Psi}^{\ast}),K_{\Theta}\right\}  .
\]

\end{proof}

Now we can prove the main Theorem \ref{thm2}. We restate this theorem here.

\begin{theorem}
\label{range} A closed subspace $N$ of $H_{E}^{2}\oplus H_{F}^{2}$ is an
invariant subspace of $S_{E}\oplus S_{F}^{*},$ if and only if $N=Span\left\{
R(V_{\Phi}),K_{\Theta}\right\}  $ where $\Phi\in\Gamma_{E\oplus F}^{\infty}$
is zero or partial isometry-valued with $\Phi(z)^{\ast}\Phi(z)=I_{E_{0}}$ a.e.
$z \in\mathbb{T}$ and $E_{0}\subset E\oplus F,$ and $\Theta\in H^{\infty
}_{B(F_{1},F)}$ is zero or left inner, right extremal.
\end{theorem}

\begin{proof}
Since $\overline{R(V_{\Phi})}$ and $K_{\Theta}$ are invariant subspaces of
$S_{E}\oplus S_{F}^{*},$ we only need to show that the invariant subspaces of
$S_{E}\oplus S_{F}^{*}$ can be expressed by $Span\left\{  R(V_{\Phi
}),K_{\Theta}\right\}  $.

Let $U_{0}:E\oplus F\rightarrow F\oplus E$ be the unitary operator defined by
\[
U_{0}(e\oplus f)=f\oplus e,e\in E,f\in F.
\]
We can extend $U_{0}$ to the unitary operator $U:H_{E}^{2}\oplus H_{F}%
^{2}\rightarrow H_{F}^{2}\oplus H_{E}^{2}$ defined by
\[
U(g\oplus h)=h\oplus g,g\in H_{E}^{2},h\in H_{F}^{2}.
\]
We can also write $U=T_{U_{0}},$ where $U_{0}$ denote the constant
operator-valued function. Then $S_{E}^{\ast}\oplus S_{F}=U^{\ast}(S_{F}\oplus
S_{E}^{\ast})U.$ Thus $M_{2}$ is an invariant subspace of $S_{F}\oplus
S_{E}^{\ast}$ if and only if $M_{2}=UM_{1}$ where $M_{1}$ is an invariant
subspace of $S_{E}^{\ast}\oplus S_{F}.$ By the previous lemma, $M_{1}%
=Span\left\{  R(W_{\Psi}^{\ast}),K_{\Theta}\right\}  ,$ where $\Psi= 0$ or
\[
W_{\Psi}=\left[
\begin{array}
[c]{cc}%
H_{C }^{\ast} & T_{A }^{\ast}\\
H_{D }^{\ast} & T_{B }^{\ast}%
\end{array}
\right]  ,\Psi(z)=\left[
\begin{array}
[c]{cc}%
C(z) & D(z)\\
A(z) & B(z)
\end{array}
\right]  :E\oplus F\rightarrow E\oplus F,
\]
with $A ,B \in H^{\infty},$ $C ,D \in L^{\infty}$, $\Psi(z)^{\ast}%
\Psi(z)=I_{E_{0}}$ a.e. $z \in\mathbb{T}$ and $E_{0}\subset E\oplus F.$ Then%
\[
W_{\Psi}^{\ast}=\left[
\begin{array}
[c]{cc}%
H_{C } & H_{D }\\
T_{A } & T_{B }%
\end{array}
\right]
\]
and $UR(W_{\Psi}^{\ast})=R(V_{\Phi})$ where $\Phi\in\Gamma_{F\oplus E}%
^{\infty}$%
\[
V_{\Phi}=\left[
\begin{array}
[c]{cc}%
T_{B } & T_{A }\\
H_{D } & H_{C }%
\end{array}
\right]  ,\Phi(z)=\left[
\begin{array}
[c]{cc}%
B(z) & A(z)\\
D(z) & C(z)
\end{array}
\right]  :F\oplus E\rightarrow F\oplus E
\]
and $\Phi(z)=U_{0}^{\ast}\Psi(z)U_{0},$ $\Phi(z)^{\ast}\Phi(z)=U_{0}^{\ast
}\Psi(z)^{\ast}\Psi(z)U_{0}=I_{E_{1}},$ where $E_{1}=U_{0}^{\ast}E_{0}.$ Here
there is an abuse of the notation $V_{\Phi}$ as an operator on $H_{F}%
^{2}\oplus H_{E}^{2}$ instead of its definition as an operator on $H_{E}%
^{2}\oplus H_{F}^{2}.$

Similarly, since $K_{\Theta}\subset H_{E}^{2}\oplus\left\{  0\right\}  ,$
$UK_{\Theta}\subset\left\{  0\right\}  \oplus H_{E}^{2}$ and $UK_{\Theta
}=\left\{  0\right\}  \oplus K_{\Theta}.$ The result now follows by
interchanging $E$ and $F.$
\end{proof}

To answer the uniqueness question for representing invariant subspaces of
$S_{E}\oplus S_{F}^{\ast}$ as $Span$ $\left\{  R(V_{\Phi}),K_{\Theta}\right\}
,$ we state the following question.

\begin{question}
Let $N=Span\left\{  R(V_{\Phi_{i}}),K_{\Theta}\right\}  $ where $\Phi_{i}\in
L_{B(E_{i},E\oplus F)}^{\infty}$ is isometry-valued, $E_{i} \subset E \oplus
F$ for $i=1,2$ and $\Theta\in H^{\infty}_{B(F_{1},F)}$ is zero or left inner,
$F_{1} \subset F$, what is the relationship between $\Phi_{1}$ and $\Phi_{2}?$
\end{question}

\section{Applications}

When $\dim E = 1, \dim F < \infty$ or $\dim E<\infty, \dim F =1$, we do a more
detailed analysis to show more precise pictures of invariant subspaces of
$S_{E}\oplus S_{F}^{\ast}.$ This analysis also demonstrates that there are
still operator and function theoretic questions concerning the fine structures
of invariant subspaces $S_{E}\oplus S_{F}^{\ast}$ in general. We use lower
case letters $a, b, \ldots$ to denote scalar functions. We first discuss the
case $\dim E=\dim F=1$ which gives a different proof of Timotin's theorem.

\

\begin{proof}
[Proof of Theorem \ref{Timotinthm}]Assume $\dim E=\dim F=1.$ Let $N$ be an
invariant subspace of $S_{E}\oplus S_{F}^{\ast}.$ There are three cases.

\textbf{Case 1} $N$ is type II. By Corollary \ref{costant1} $N_{3} \supset
L^{2}_{E}$ and $N = \{0\} \oplus H^{2}_{F}$, or $N = \{0\} \oplus K_{\theta}$,
where $K_{\theta}= H^{2}_{F} \ominus\theta H^{2}_{F}$, $\theta$ inner
function. In this case, $N$ is splitting.

\textbf{Case 2} $N$ is type I with $N=\left(  L_{E}^{2}\oplus H_{F}%
^{2}\right)  \ominus J_{2}N_{3},$ where $N_{3}=UH_{E}^{2}.$ Then%
\[
U(z)=\left[
\begin{array}
[c]{c}%
za(\overline{z})\\
c(z)
\end{array}
\right]  ,a,c\in H^{\infty}.
\]
By Proposition \ref{ab1}, $\text{rank}\left[  c(z)\right]  =\dim E - \dim
E=0.$ So $c =0$ and $a $ is an inner function. Thus $N$ is splitting. In fact,
$N_{3} = za(\overline{z})H^{2}_{E}$ and $N = aH^{2}_{E} \oplus H^{2}_{F}$. Or
one applies Theorem \ref{type1k} to conclude this.

\textbf{Case 3} $N$ is type I with $N=\left(  L_{E}^{2}\oplus H_{F}%
^{2}\right)  \ominus J_{2}N_{3},$ where $N_{3}=U H_{E\oplus F}^{2}$ and $U $
is unitary-valued with
\[
U(z) = \left[
\begin{array}
[c]{cc}%
za(\overline{z}) & zb(\overline{z})\\
c(z) & d(z)
\end{array}
\right]  , a,b,c,d \in H^{\infty}%
\]
By Theorem \ref{mainr}, $N=R(V_{\Phi_{1}})$ where $\Phi_{1}$ is defined in
(\ref{v1}). Note that if $\Phi_{1}$ is given by (\ref{v1}) with $a,b,c,d \in
H^{\infty}$ and $\Phi_{1}$ is unitary-valued, then $V_{\Phi_{1}}$ is a partial
isometry and $R(V_{\Phi_{1}})$ is an invariant subspace of $S_{E} \oplus
S_{F}^{*}$.

The splitting invariant subspaces of $S_{E}\oplus S_{F}^{*}$ are of the form
$N = \theta_{1}H^{2}_{E} \oplus K_{\theta_{2}}$, $\theta_{1}, \theta_{2}$
inner functions or zero. If $\theta_{1} = 0$, then it is Case 1. If
$\theta_{1} \neq0, \theta_{2} = 0$, then it is Case 2. So $N$ is splitting in
Case 3 if and only if $N = \theta_{1}H^{2}_{E} \oplus K_{\theta_{2}}$ with
$\theta_{1}, \theta_{2}$ inner functions. Note that $N = \theta_{1}H^{2}_{E}
\oplus K_{\theta_{2}}$, $\theta_{1}, \theta_{2}$ inner functions, if and only
if $N_{3} = U H_{E\oplus F}^{2}$ with
\[
U(z) = \left[
\begin{array}
[c]{cc}%
za(\overline{z}) & zb(\overline{z})\\
c(z) & d(z)
\end{array}
\right]  = \left[
\begin{array}
[c]{cc}%
z\theta_{1}(\overline{z}) & 0\\
0 & \theta_{2}(z)
\end{array}
\right]  W,
\]
where $W$ is a constant unitary on $E\oplus F = \mathbb{C}^{2}$. Thus $N$ is
splitting in this case if and only if $a$ and $b$ are proportional, i.e. there
exists $(\alpha, \beta) \in\mathbb{C}^{2} \backslash\{(0,0)\}$ such that
$\alpha a + \beta b =0$.

The representation of non-splitting $N=R(V_{\Phi_{1}})$ is unique in the
following way. If there is another $\Phi_{1}^{\prime}$ such that
$N=R(V_{\Phi_{1}^{\prime}}),$ where $\Phi_{1}^{\prime}$ satisfies the same
condition of $\Phi_{1}$, then there exists a $2\times2$ constant unitary
matrix $W$ such that $\Phi_{1}^{\prime} =\Phi_{1} W.$ The proof is complete.
\end{proof}

The above theorem identifies all the non-splitting invariant subspaces of
$S_{E} \oplus S_{F}^{*}$ when $\dim E = \dim F =1$. In general we may ask
whether there exists non-splitting invariant subspace of $S_{E} \oplus
S_{F}^{*}$. The following are examples modified from \cite[Example 7.3]%
{CRoss}. When $E$ and$\ F$ are finite dimensional, let%
\begin{align}\label{modifiedexamp}
N=\left\{  \left[
\begin{array}
[c]{cccccc}%
f & \cdots & f & f(0) & \cdots & f(0)
\end{array}
\right]  ^{tr}:f\in H^{2}\right\}  ,
\end{align}
where the first $m =\dim E$ entries are $f$, the last $n=\dim F$ entries are
$f(0)$, then $N$ is a non-splitting invariant subspace. When $E$ and$\ F$ are
infinite dimensional, let%
\[
N=\left\{  \left[
\begin{array}
[c]{cccccccc}%
f & f /2 & f /2^{2} & \cdots & f(0) & f(0)/2 & f(0)/2^{2} & \cdots
\end{array}
\right]  ^{tr}:f\in H^{2}\right\}  ,
\]
then $N$ is a non-splitting invariant subspace.

The following result provides a positive answer for Question \ref{invaiantk}
in the case $\dim E=1$. When $\dim E=1$, it is known that there exists $a\in
L^{\infty}$ such that $\ker(H_{a})=\left\{  0\right\}  .$ In fact, write
$a(z)=f(z)+\overline{zg(z)},$ where $f,g\in H^{2},$ then $\ker(H_{a}%
)=\left\{  0\right\}  $ if and only if $Span\left\{  H_{\overline
{a(\overline{z})}}z^{n}:n\geq0\right\}  =Span\left\{  S_{E}^{\ast n}%
g:n\geq 0\right\}  =H^{2},$ if and only if
$g$ is a cyclic vector of $S_{E}^{\ast}$.
The cyclic vectors of $S_E^*$ are described analytically in \cite{DSS70}.
For example, let
$\overline{a}(z)=\sum_{j=1}^{\infty}\frac{c_{j}}{\overline{\lambda_{j}}}%
\frac{1}{1-\overline{\lambda_{j}}z}$, where $\{\lambda_{j}\}\subset\mathbb{D}$
is distinct and $\lambda_{j}\rightarrow0$, and $c_{j}>0$ are chosen to satisfy
$\sum_{j=1}^{\infty}\left\vert \frac{c_{j}}{\lambda_{j}}\right\vert \frac
{1}{1-|\lambda_{j}|}<\infty$, then $\ker(H_{a})=\left\{  0\right\}  $ (or see
the argument in \cite[Theorem 2.5]{LR15}).

\begin{theorem}
\label{typeiikernel} Assume $\dim E=1$. Let $N$ be an invariant subspace of
$S_{E}\oplus S_{F}^{\ast}$. Then $N=\ker(W_{\Psi})$ for some $\Psi\in
\Lambda_{E\oplus F}^{\infty}.$ If Furthermore, $\dim F < \infty$ and $N$ is a
type I invariant subspace, then there is a partial isometry $W_{\Psi}$ such
that $N=\ker(W_{\Psi})$.
\end{theorem}

\begin{proof}
If $N$ is type I, then by Theorem \ref{type1k} $N=\ker(W_{\Psi})$ for some
isometry-valued function $\Psi$. If also $\dim F < \infty$, then Corollary
\ref{maink1} implies that we can take $W_{\Psi}$ to be a partial isometry so
that $N=\ker(W_{\Psi})$.

If $N$ is type II, since $\dim E =1$, Corollary \ref{costant1} implies that $N
= \{0\} \oplus M$ with $M = H^{2}_{F}$ or $H^{2}_{F} \ominus\Theta
H^{2}_{F_{1}}$ ($F_{1} \subset F, \Theta$ left inner) an invariant subspace of
$S_{F}^{*}$ on $H^{2}_{F}$. So $N$ is splitting. Let $\overline{a} \in
H^{\infty}$ be such that $\ker(H_{a}^{\ast})=\left\{  0\right\}  $. If $M =
H^{2}_{F}$, let $C = 0$, if $H^{2}_{F} \ominus\Theta H^{2}_{F_{1}}$, extend
$\Theta$ to be zero on $H^{2}_{F_{1}^{\perp}}$ and let $C = \Theta$, then
$\ker(T_{C}^{*}) = M$. Now let
\[
W_{\Psi}=\left[
\begin{array}
[c]{ccc}%
H_{a}^{\ast} & 0  \\
0 & T_{C}^{\ast}
\end{array}
\right]  ,\Psi(z)=\left[
\begin{array}
[c]{ccc}%
a(z) & 0 \\
0 & C(z)
\end{array}
\right]  .
\]
Then $N=\ker(W_{\Psi})$.
\end{proof}

If $\dim E=1,\dim F<\infty$, and $N$ is a non-splitting invariant subspace,
then from the above proof, we see that $N$ is type I and there is a partial
isometry $W_{\Psi}$ such that $N=\ker(W_{\Psi})$.

\

Now we consider the case $\dim E < \infty, \dim F =1$. We need the following lemma.

\begin{lemma}
\label{rangetypeii} Assume $E$ and $F$ are finite dimensional. Let
$N_{3}=UH_{E_{0}}^{2}\oplus\Omega L_{E_{2}}^{2}$ be such that $zH_{E}%
^{2}\subset N_{3}\subset L_{E}^{2}\oplus H_{F}^{2},$ where $E_{0}\perp
E_{2},E_{0}\oplus E_{2}=E\oplus F$, $U\in L_{B(E_{0},E\oplus F)}^{\infty
},\Omega\in L_{B(E_{2},E)}^{\infty}$ are isometry-valued. Let $N=\left(
L_{E}^{2}\oplus H_{F}^{2}\right)  \ominus J_{2}N_{3}$. Then there is a partial
isometry $V_{\Phi}$ such that $N=R(V_{\Phi})$.
\end{lemma}

\begin{proof}
Write
\[
U(z)=\left[
\begin{array}
[c]{c}%
U_{E}(z)\\
U_{F}(z)
\end{array}
\right]  =\left[
\begin{array}
[c]{c}%
zA(\overline{z})\\
C(z)
\end{array}
\right]  :E_{0}\rightarrow E\oplus F.
\]
Since $E$ and $F$ are finite dimensional, $E_{0}\oplus E_{2}=E\oplus F$, by
Proposition \ref{type2condition}, $\left[
\begin{array}
[c]{cc}%
U & \Omega
\end{array}
\right]  $ is unitary-valued, $A,C\in H^{\infty}$ and $\text{rank}C(z)=\dim
F$. Note that $\Omega\in L_{B(E_{2},E)}^{\infty}$, we thus have
$C(z)C(z)^{\ast}=I_{F}$ a.e. $z\in\mathbb{T}$. By Theorem \ref{type2k} and the
remark after that, $N=\ker(W_{\Psi})\cap(\Theta H_{E_{1}}^{2}\oplus H_{F}%
^{2})$, where $\Psi$ is given by (\ref{u2}), $E_{1}\subset E$, and $\Theta\in
H_{B(E_{1},E)}^{\infty}$ is zero or left inner such that $\Omega^{\ast}%
\Theta=0.$ Also $A=\Theta\Gamma$ for some $\Gamma\in H_{B(E_{0},E_{1}%
)}^{\infty}$.

If $\Theta=0$, then $A=0$. Lemma \ref{condition} (iii) then implies that
$E_{2}=E,$ $\Omega L_{E_{2}}^{2}=L_{E}^{2}$. So $\dim E_{0}=F$, $C$ is
unitary-valued and $N_{3}=CH_{E_{0}}^{2}\oplus L_{E}^{2}$. Thus $N=H_{F}%
^{2}\ominus CH_{E_{0}}^{2}$. Let
\[
V_{\Phi}=\left[
\begin{array}
[c]{ccc}%
0 & 0  \\
H_{C(\overline{z})} & 0
\end{array}
\right]  ,\Phi(z)=\left[
\begin{array}
[c]{ccc}%
0 & 0  \\
C(\overline{z}) & 0
\end{array}
\right]  .
\]
By Lemma \ref{partial1} $V_{\Phi}$ is a partial isometry and $N=R(V_{\Phi})$.

If $\Theta\neq0$, then $A=\Theta\Gamma$ for some $\Gamma\in H_{B(E_{0},E_{1}%
)}^{\infty}$. By Theorem \ref{type2k},
\[
N=\left\{  \left[
\begin{array}
[c]{c}%
\Theta f\\
g
\end{array}
\right]  \in H_{E}^{2}\oplus H_{F}^{2}:\left[
\begin{array}
[c]{c}%
f\\
g
\end{array}
\right]  \in\ker(W_{\Psi_{1}})\subset H_{E_{1}}^{2}\oplus H_{F}^{2}\right\}
,
\]
where $\Psi_{1}\in L_{B(E_{0},E_{1}\oplus F)}^{\infty}$ is isometry-valued
with
\[
\Psi_{1}(z)=\left[
\begin{array}
[c]{c}%
\Gamma(\overline{z})\\
C(z)
\end{array}
\right]  :E_{0}\rightarrow E_{1}\oplus F.
\]
Now we show $\Psi_{1}$ is unitary-valued. Since $\text{rank}C(z)=\dim F$, it
is enough to show $\text{rank}\Gamma(z)\geq\dim E_{1}$. Since $\left[
\begin{array}
[c]{cc}%
U & \Omega
\end{array}
\right]  $ is unitary-valued, we have $\text{rank}A(z)=\dim E-\dim E_{2}$. By
$A=\Theta\Gamma$, we obtain $\text{rank}\Gamma(z)\geq\text{rank}A(z)=\dim
E-\dim E_{2}$. Note that $\Omega^{\ast}(\overline{z})\Theta(z)=0$, we have $\dim
E_{1}\leq\dim E-\dim E_{2}$. Thus $\text{rank}\Gamma(z)\geq\dim E_{1}$ and
$\Psi_{1}$ is unitary-valued. Let $\Phi_{1}(z)=\Psi_{1}(\overline{z})$, then
Theorem \ref{mainr} or Lemma \ref{partial1} implies that $\ker(W_{\Psi_{1}%
})=R(V_{\Phi_{1}})$. Note that
\[
R(V_{\Phi_{1}})=\left[
\begin{array}
[c]{c}%
T_{\Gamma}\\
H_{C(\overline{z})}%
\end{array}
\right]  H_{E_{0}}^{2}=\left\{  \left[
\begin{array}
[c]{c}%
\Gamma k\\
H_{C(\overline{z})}k
\end{array}
\right]  :k\in H_{E_{0}}^{2}\right\}  .
\]
Then
\[
N=\left\{  \left[
\begin{array}
[c]{c}%
\Theta\Gamma k\\
H_{C(\overline{z})}k
\end{array}
\right]  :k\in H_{E_{0}}^{2}\right\}  =\left\{  \left[
\begin{array}
[c]{c}%
Ak\\
H_{C(\overline{z})}k
\end{array}
\right]  :k\in H_{E_{0}}^{2}\right\}  .
\]
Let
\[
V_{\Phi}=\left[
\begin{array}
[c]{cc}%
T_{A} & 0\\
H_{C(\overline{z})} & 0
\end{array}
\right]  ,\Phi(z)=\left[
\begin{array}
[c]{cc}%
A(z) & 0\\
C(\overline{z}) & 0
\end{array}
\right]  :E_{0}\oplus E_{0}^{\perp}\rightarrow E\oplus F.
\]
Then $N=R(V_{\Phi})$, and by Lemma \ref{partial1} $V_{\Phi}$ is a partial isometry.
\end{proof}

Now we are ready to prove the following version of Theorem \ref{mandonerange2}
which is related to Question \ref{invaiantr} in the case $\dim E<\infty, \dim
F=1.$

\begin{theorem}
\label{mandonerange} Assume $\dim E=m<\infty$ and $\dim F=1.$ Let $N$ be an
invariant subspace of $S_{E}\oplus S_{F}^{\ast}$. Then $N=\overline{R(V_{\Phi
})}$ for some $\Phi\in\Gamma_{E\oplus F}^{\infty}$. Furthermore, if $N$ is a
non-splitting invariant subspace of $S_{E}\oplus S_{F}^{\ast}$, then there is
a partial isometry $V_{\Phi}$ such that $N=R(V_{\Phi})$.
\end{theorem}

\begin{proof}
We split the proof into the following cases.

\textbf{Case 1} $N$ is type I. Assume $N$ is type I with $N=\left(  L_{E}%
^{2}\oplus H_{F}^{2}\right)  \ominus J_{2}N_{3},$ where $N_{3}=UH_{E_{0}}^{2}$
and $U \in L^{\infty}_{B(E_{0},E\oplus F)}$ is isometry-valued. Write
\begin{equation}
\label{dimmandone}U(z)=\left[
\begin{array}
[c]{c}%
U_{E}(z)\\
U_{F}(z)
\end{array}
\right]  =\left[
\begin{array}
[c]{c}%
zA(\overline{z})\\
C(z)
\end{array}
\right]  :E_{0} \rightarrow E \oplus F, \Phi(z)=\left[
\begin{array}
[c]{c}%
A(z)\\
C(\overline{z})
\end{array}
\right]  .
\end{equation}
Then $A, C \in H^{\infty}$ and $\dim E_{0} \geq\dim E$. We may suppose $E
\subset E_{0}$, and we have the following two subcases.

\textbf{Case 1a} $N$ is type I with $N=\left(  L_{E}^{2}\oplus H_{F}%
^{2}\right)  \ominus J_{2}N_{3},$ where $N_{3}=UH_{E_{0}}^{2}, E_{0} = E\oplus
F$ and $U $ is unitary-valued. By Theorem \ref{mainr}, $V_{\Phi}$ is a partial
isometry and $N=R(V_{\Phi})$, where $\Phi$ is given by (\ref{dimmandone}).

\textbf{Case 1b} $N$ is type I with $N=\left(  L_{E}^{2}\oplus H_{F}%
^{2}\right)  \ominus J_{2}N_{3},$ where $N_{3}=UH_{E}^{2}$. By Proposition
\ref{ab1}, $C=0.$ Thus $A$ is unitary-valued, $N_{3} = zA(\overline{z}%
)H^{2}_{E}$, and $N =AH^{2}_{E} \oplus H^{2}_{F}$ is splitting. Let
$\overline{\alpha} \in H^{\infty}$ be such that the closure of $R(H_{\alpha})$
is $H_{F}^{2}$, and let
\[
V_{\Phi}=\left[
\begin{array}
[c]{ccc}%
T_{A} & 0  \\
0 & H_{\alpha}
\end{array}
\right]  , \Phi(z)=\left[
\begin{array}
[c]{ccc}%
A(z) & 0  \\
0 & \alpha(z)
\end{array}
\right]  ,
\]
then $N=\overline{R(V_{\Phi})}$.

\textbf{Case 2} $N$ is type II. Assume $N$ is type II with $N=\left(
L_{E}^{2}\oplus H_{F}^{2}\right)  \ominus J_{2}N_{3},$ where, by Proposition
\ref{type2condition}, $N_{3}=UH_{E_{0}}^{2}\oplus\Omega L_{E_{2}}^{2}$ with
$E_{2}\perp E_{0},\Omega\in L_{B(E_{2},E)}^{\infty}$ isometry-valued and $U\in
L_{B(E_{0},E\oplus F)}^{\infty}$ being zero or isometry-valued. If $U=0$, then
$N_{3}=L_{E}^{2}$ and $N=H_{F}^{2}$ is splitting. Thus as in Case 1b, there is
$\Phi$ such that $N=\overline{R(V_{\Phi})}$. Now we assume $U$ is
isometry-valued and $\dim E_{0}=n$. Then $1\leq n\leq m$. Since $\dim E+\dim
F=m+1$, $U(z)E_{0}\perp\Omega L_{E_{2}}^{2}$, by Lemma \ref{condition} (iii),
$\max\{1,m-n\}\leq\dim E_{2}\leq m+1-n$.

If $n < m$ and $\dim E_{2}=m-n$, by Proposition \ref{type2condition},
$\text{rank}C(z)=0$ and $C=0$. Therefore $N_{3}\subset L_{E}^{2}$, and
$N=M\oplus H_{F}^{2}$ is splitting, where $M$ is some invariant subspace of
$S_{E}$. As in Case 1b, there is $\Phi$ such that $N=\overline{R(V_{\Phi})}$.

If $n < m$ and $\dim E_{2}=m-n+1$, then $\dim E_{0}+\dim E_{2}=m+1$ and
$E_{0}\oplus E_{2}=E\oplus F$. Thus Lemma \ref{rangetypeii} implies that there
is a partial isometry $V_{\Phi}$ such that $N=R(V_{\Phi})$.

If $n = m$, then $\dim E_{2} = 1$ and $E_{0}\oplus E_{2}=E\oplus F$. So by
Lemma \ref{rangetypeii} there is a partial isometry $V_{\Phi}$ such that
$N=R(V_{\Phi})$. The proof is complete.
\end{proof}

When $\dim E<\infty$ and $1<\dim F<\infty$, we don't know whether every
invariant subspace of $S_{E}\oplus S_{F}^{\ast}$ is of the form $\overline
{R(V_{\Phi})}$ with $\Phi\in\Gamma_{E\oplus F}^{\infty}$. For the
non-splitting invariant subspaces examples in (\ref{modifiedexamp}), it is not hard to see
that they are of form $R(V_{\Phi}).$ Assume $\dim E=1,\dim F=2$. Let
\[
N=\left\{  \left[
\begin{array}
[c]{cccccc}%
f & f(0) & f(0)
\end{array}
\right]  ^{tr}:f\in H^{2}\right\}  .
\]
Then $N=R(V_{\Phi}),$ where
\[
V_{\Phi}=\frac{1}{\sqrt{3}}\left[
\begin{array}
[c]{c}%
T_{a}\\
H_{\overline{z}}\\
H_{\overline{z}}%
\end{array}
\right]  ,a(z)=1,\Phi=\frac{1}{\sqrt{3}}\left[
\begin{array}
[c]{c}%
a\\
\overline{z}\\
\overline{z}%
\end{array}
\right]  \text{ is isometry-valued.}%
\]

We end this section with the following question which is a more precise
version of Question \ref{invaiantr}.

\begin{question}
Assume $\dim E < \infty$ and $1<\dim F <\infty$. Let $N$ be a non-splitting
invariant subspace of $S_{E}\oplus S_{F}^{\ast}$. Then is $N = \overline
{R(V_{\Phi})}$ for some $\Phi\in\Gamma_{E\oplus F}^{\infty}$?
\end{question}

\

\noindent \textbf{Acknowledgement}:
S. Luo was supported by NNSFC (\# 11701167).

\bigskip\bigskip

Caixing Gu

Department of Mathematics, California Polytechnic State University,

San Luis O-bispo, CA 93407, USA

E-mail address: cgu@calpoly.edu \bigskip\bigskip

Shuaibing Luo

School of Mathematics, and Hunan Provincial Key Laboratory of Intelligent
information processing and Applied Mathematics, Hunan University,

Changsha, 410082, PR China

E-mail address: sluo@hnu.edu.cn

\bigskip\bigskip

Mathematics Subject Classification (2010). 47A15, 47B37, 47B35

Keywords. Invariant subspace, unilateral shift, Hankel operator, Toeplitz operator.

\end{document}